\documentclass[reqno, 11pt]{amsart}
\usepackage{amsmath}
\usepackage{amsthm}
\usepackage{amsfonts}
\usepackage{amssymb}
\usepackage{mathrsfs}
\usepackage{graphicx}

\newcommand{\mb}{\mathbf}
\newcommand{\mc}{\mathcal}

\renewcommand{\Re}{\mathrm{Re}\,}

\newcommand{\rg}{\mathrm{rg}\,}
\newcommand{\N}{\mathbb{N}}
\newcommand{\R}{\mathbb{R}}
\newcommand{\C}{\mathbb{C}}
\newcommand{\Z}{\mathbb{Z}}

\newcommand \be{\begin{equation}}
\newcommand \ee{\end{equation}}

\newtheorem{lemma}{Lemma}[section]
\newtheorem{proposition}[lemma]{Proposition}
\newtheorem{theorem}[lemma]{Theorem}

\theoremstyle{remark}
\newtheorem{remark}[lemma]{Remark}

\theoremstyle{definition}

\numberwithin{equation}{section}

\DeclareSymbolFontAlphabet{\mathrsfs}{rsfs}
\DeclareMathAlphabet{\mathcal}{OMS}{cmsy}{m}{n}
\newcommand{\scri}{\mathrsfs{I}}

\title[Nondispersive decay]{Nondispersive decay for the cubic wave equation}

\author{Roland Donninger}
\address{\'Ecole Polytechnique F\'ed\'erale de Lausanne, 
Department of Mathematics, Station 8, CH-1015 Lausanne, Switzerland}
\email{roland.donninger@epfl.ch}

\author{An{\i}l Zengino\u{g}lu}
\address{Theoretical Astrophysics, M/C 350--17, California Institute of Technology, 
Pasadena, California 91125, USA}
\email{anil@caltech.edu} 

\thanks{The authors would like to thank the Erwin Schr\"odinger Institute for Mathematical
Physics (ESI) in Vienna for hospitality during the workshop ``Dynamics of General Relativity:
Black Holes and Asymptotics''
where this work was initiated. AZ is supported by the NSF grant PHY-106881 and by a Sherman Fairchild Foundation grant to Caltech.}

\begin{document}

\begin{abstract}
We consider the hyperboloidal initial value problem for the cubic focusing wave equation
\[ (-\partial_t^2+\Delta_x)v(t,x)+v(t,x)^3=0,\quad x\in \R^3. \]
Without symmetry assumptions, we prove the existence of a co-dimension 4 Lipschitz manifold of initial data that lead to global
solutions in forward time which do not scatter to free waves.
More precisely, for any $\delta\in (0,1)$,
we construct solutions with
the asymptotic behavior
\[ \|v-v_0\|_{L^4(t,2t)L^4(B_{(1-\delta)t})}\lesssim t^{-\frac12+} \]
as $t\to \infty$ where $v_0(t,x)=\frac{\sqrt 2}{t}$ and $B_{(1-\delta)t}:=\{x\in \R^3: |x|<(1-\delta)t\}$.
\end{abstract}

\maketitle 

\section{Introduction}

We consider the cubic
focusing wave equation
\begin{equation}
\label{eq:main} 
(-\partial_t^2+\Delta_x)v(t,x)+v(t,x)^3=0
\end{equation}
in three spatial dimensions.
Eq.~\eqref{eq:main} admits the conserved energy
\[ E(v(t,\cdot),v_t(t,\cdot))=\tfrac12 
\|(v(t,\cdot),v_t(t,\cdot))\|_{\dot H^1\times L^2(\R^3)}^2-\tfrac14 \|v(t,\cdot)\|_{L^4(\R^3)}^4\,, \]
and it is well-known that solutions with small $\dot H^1\times L^2(\R^3)$-norm
exist globally and scatter to zero \cite{Str81, MocMot85, MocMot87, Pec88}, 
whereas solutions with negative energy blow up 
in finite time \cite{Gla73, Lev74}.
There exists an explicit blowup solution $\tilde v_T(t,x)=\frac{\sqrt 2}{T-t}$,
which describes a stable blowup regime \cite{DonSch12}
and the blowup speed (but not the profile) of any blowup solution \cite{MerZaa05},
see also \cite{BizonChmajTabor04} for numerical work.
By the time translation and reflection symmetries of Eq.~\eqref{eq:main}
we obtain from $\tilde v_T$ the explicit solution $v_0(t,x)=\frac{\sqrt{2}}{t}$,
which is now global for $t\geq 1$ and decays in a nondispersive manner.
However, in the context of the standard Cauchy problem, 
where one prescribes data at $t=t_0$ for some $t_0$ and considers
the evolution for $t\geq t_0$, the role of $v_0$ for the study of global solutions
is unclear because $v_0$ has
infinite energy.
In the present paper we argue that this is not a defect of the solution $v_0$ but
rather a problem of the usual viewpoint concerning the Cauchy problem. 
Consequently, we study a different type of initial value problem for Eq.~\eqref{eq:main} 
where we prescribe data on a spacelike hyperboloid. 
In this formulation there exists a different ``energy'' which is finite for $v_0$.

Hyperboloidal initial value formulations have many advantages over the standard
Cauchy problem and are well-known
in numerical and mathematical relativity \cite{Eardley:1978tr, Fried83, Frau04, Zen07}.
However, in the mathematical literature on wave equations in flat spacetime, hyperboloidal
initial value formulations are less common (with notable exceptions such as \cite{Christo86}). 
We provide a thorough discussion of hyperboloidal methods in Section \ref{sec:geo}, where we argue 
that the hyperboloidal initial value problem is natural for hyperbolic equations in view of the underlying 
Minkowski geometry.

To state our main result, we consider a foliation of the future of
the forward null cone emanating from the origin by spacelike 
hyperboloids 
\[ \Sigma_T:=\left \{(t,x)\in \R\times \R^3: t=-\tfrac{1}{2T}+\sqrt{\tfrac{1}{4T^2}+|x|^2} \right \}\,, \]
where $T \in (-\infty,0)$. Each $\Sigma_T$ is parametrized by 
\[ \Phi_T: B_{|T|}\subset \R^3 \to \R^4,\quad 
\Phi_T(X)=\left (-\frac{T}{T^2-|X|^2},\frac{X}{T^2-|X|^2}\right )\,, \]
where $B_R:=\{X\in \R^3: |X|<R\}$ for $R>0$.
The ball $B_{|T|}$ shrinks in time as $T\to 0-$, but its image under $\Phi_T$ is 
an unbounded spacelike hypersurface in Minkowski space. 
The transformation $(T,X)\mapsto \Phi_T(X)$ has also been used by Christodoulou to study semilinear wave equations 
in \cite{Christo86} and is known as the Kelvin inversion \cite{Tao}. 
Note that in four-dimensional notation it can be
written as $X^\mu\mapsto -X^\mu/(X^\nu X_\nu)$ (up to a sign in the zero component). 
To illustrate the resulting initial value problem, 
we plot the spacelike hyperboloids $\Sigma_T$ for various values of 
$T\in (-\infty,0)$ in a spacetime diagram (left panel) and in a Penrose diagram (right panel) 
in Fig.~\ref{fig:Kelvin} along with a null surface emanating from the origin. 
In our formulation of the initial value problem we prescribe data on the hypersurface $\Sigma_{-1}$
and consider the future development.
We refer the reader to Section \ref{sec:geo} for a discussion 
on hyperboloidal foliations and their relation to wave equations.

\begin{figure}[ht]
\center
\includegraphics[width=0.7\textwidth]{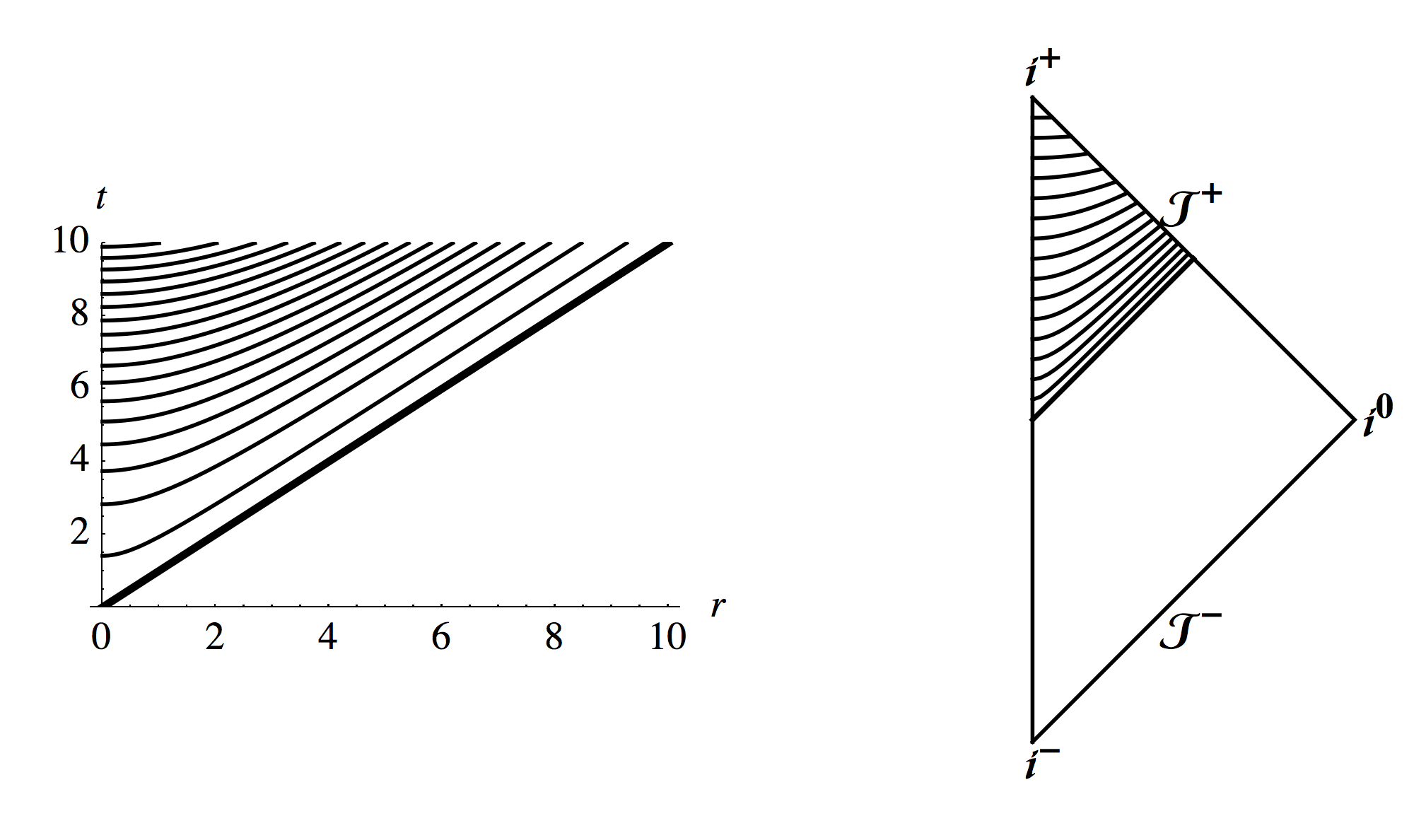}
\caption{The spacelike hyperboloids $\Sigma_T$ in a spacetime diagram (left panel) 
and a Penrose diagram (right panel) together with the null surface emanating from the 
origin (thick line with 45 degrees to the horizontal). Compare Fig.~\ref{fig:PenroseStandard}.
\label{fig:Kelvin}}
\end{figure}

We define a differential operator $\nabla_n$ by
\[ \frac{(\nabla_n v)\circ \Phi_T(X)}{T^2-|X|^2}=\partial_T \frac{(v \circ \Phi_T)(X)}{T^2-|X|^2}\,, \]
which one should think of as the normal derivative to the surface $\Sigma_T$ (although this
is not quite correct due to the additional factor $\frac{1}{T^2-|X|^2}$).
Explicitly, we have
\[ \nabla_n v(t,x)=(t^2+|x|^2)\partial_t v(t,x)+2tx^j\partial_j v(t,x)
+2tv(t,x). \]
On each leaf $\Sigma_T$ we define the norms
\begin{align} 
\|v\|_{L^2(\Sigma_T)}^2&:=\int_{B_{|T|}}
\left |\frac{v\circ \Phi_T(X)}{T^2-|X|^2} \right |^2 dX\,, \nonumber \\
\|v\|_{\dot H^1(\Sigma_T)}^2&:=\int_{B_{|T|}}\left |
\nabla_X \frac{v \circ \Phi_T(X)}{T^2-|X|^2}\right |^2 dX\,,
\end{align}
and we write $\|\cdot\|_{H^1(\Sigma_T)}^2=\|\cdot\|_{\dot H^1(\Sigma_T)}^2
+|T|^{-2}\|\cdot\|_{L^2(\Sigma_T)}^2$.
We emphasize that
\[ v_0 \circ \Phi_T(X)=\sqrt{2}\frac{T^2-|X|^2}{(-T)}\,, \]
and thus, $\|v_0\|_{H^1(\Sigma_T)}+\|\nabla_n v_0\|_{L^2(\Sigma_T)}\simeq |T|^{-\frac12}$.
Finally, for any subset $A\subset \R^4$ we denote 
its future domain of dependence by $D^+(A)$. 
With this notation at hand, we state our main result.

\begin{theorem}
\label{thm:main}
There exists a co-dimension 4 Lipschitz manifold $\mc M$ of functions in $H^1(\Sigma_{-1})\times 
L^2(\Sigma_{-1})$ with $(0,0)\in \mc M$ such that the following holds.
For data $(f,g) \in \mc M$ the hyperboloidal initial value problem
\[ \left \{
\begin{array}{l}
(-\partial_t^2 +\Delta_x) v(t,x)+v(t,x)^3=0 \\
v|_{\Sigma_{-1}}=v_0|_{\Sigma_{-1}}+f \\
\nabla_n v|_{\Sigma_{-1}}=\nabla_n v_0|_{\Sigma_{-1}}+g \end{array} \right. \]
has a unique solution $v$ defined on $D^+(\Sigma_{-1})$
such that 
\[ |T|^\frac12\left (\|v-v_0\|_{H^1(\Sigma_T)}+\|\nabla_n v-\nabla_n v_0\|_{L^2(\Sigma_T)}\right )
\lesssim
|T|^{\frac12-} \]
for all $T \in [-1,0)$.
As a consequence, for any $\delta \in (0,1)$, we have 
\begin{align*} 
\|v-v_0\|_{L^4(t,2t)L^4(B_{(1-\delta)t})}&\lesssim t^{-\frac12+}  
\end{align*}
as $t \to \infty$, i.e., $v$ converges to $v_0$ in a localized Strichartz sense.
\end{theorem}

Some remarks are in order.
\begin{itemize}
\item As usual, by a ``solution'' we mean a function which solves the equation in an
appropriate weak sense, not necessarily in the sense of classical derivatives.

\item The manifold $\mc M$ can be represented as a graph of a Lipschitz function.
More precisely, let $\mc H:=H^1(\Sigma_{-1})\times 
L^2(\Sigma_{-1})$ and denote by $\mc B_R(0)$
the open ball of radius $R>0$ around $0$ in $\mc H$.
We prove that there exists a decomposition 
$\mc H=\mc H_1 \oplus \mc H_2$ with $\dim \mc H_2=4$
and a function
$F: \mc H_1\cap \mc B_{\delta}(0) \to \mc H_2$ such that 
$\mc M=\{\vec u+F(\vec u): \vec u\in \mc H_1 \cap \mc B_{\delta}(0)\}$ provided
$\delta>0$ is chosen sufficiently small.
Furthermore, $F$ satisfies
\[ \|F(\vec u)-F(\vec v)\|_{\mc H}\lesssim \delta^\frac12 \|\vec u-\vec v\|_{\mc H} \]
for all $\vec u,\vec v \in \mc H_1\cap \mc B_{\delta}(0)$ and $F(\vec 0)=\vec 0$.
\item The reason for the co-dimension 4 instability of the attractor $v_0$
is the invariance of Eq.~\eqref{eq:main} under time translations and Lorentz transforms
(combined with the Kelvin inversion).
The Lorentz boosts do not destroy the nondispersive character of
the solution $v_0$ whereas the time translation does,
see the beginning of Section \ref{sec:lin} below for a more detailed discussion.
In this sense, one may say that there exists a co-dimension \emph{one} manifold of data that 
lead to nondispersive solutions. 
However, if one fixes $v_0$, as we have done in our formulation, there are 4 unstable directions.
\end{itemize}

There was tremendous recent progress in the understanding of universal properties
of global solutions to nonlinear wave equations, in particular in the energy critical case,
see e.g.~\cite{DuyKenMer12a, DuyKenMer12b, CotKenLawSch12, KenLawSch13}.
A guiding principle for all these studies is the soliton resolution conjecture, i.e., the idea 
that global solutions to nonlinear dispersive equations decouple into solitons
plus radiation as time tends to infinity. 
It is known that, in such a strict sense, soliton resolution does not hold in most cases.
One possible obstacle is the existence of global solutions which do not scatter.
Recently, the first author and Krieger constructed nonscattering solutions for the energy critical
focusing wave equation \cite{DonKri13}, see also \cite{OrtPer12} for similar
results in the context of the nonlinear Schr\"odinger equation.
These solutions are obtained by considering a rescaled ground state soliton,
the existence of which is typical for critical dispersive equations.
The cubic wave equation under consideration is energy subcritical and does not admit
solitons.
Consequently, our result is of a completely different nature.
Instead of considering moving solitons, 
we obtain the nonscattering solutions by perturbing the \emph{self-similar} solution 
$v_0(t,x)=\frac{\sqrt 2}{t}$.
This can only be done in the framework of a hyperboloidal initial value formulation
because the standard energy for the self-similar solution $v_0$ is infinite.

Another novel feature of our result is a precise description of the data
which lead to solutions that converge to $v_0$: They lie on a Lipschitz manifold
of co-dimension 4.
In this respect we believe that our result is also interesting from the perspective
of infinite-dimensional dynamical systems theory for wave equations, 
which is currently a very active field, see e.g.~\cite{KriNakSch10, KriNakSch11, KriNakSch12}.

Finally, we mention that the present work is motivated by numerical 
investigations undertaken by Bizo\'n and the second author \cite{BizZen09}. 
In particular, the conformal symmetry for the cubic wave equation 
has been used in \cite{BizZen09} to translate the (linear) stability analysis
for blowup to asymptotic results for decay.
We exploit this idea in a similar way: 
If $v$ solves Eq.~\eqref{eq:main} then $u$, defined by
\[ u(T,X)=\frac{1}{T^2-|X|^2}v\left (-\frac{T}{T^2-|X|^2},\frac{X}{T^2-|X|^2}\right )
= \frac{v \circ \Phi_T(X)}{T^2-|X|^2} , \]
solves $(-\partial_T^2+\Delta_X)u(T,X)+u(T,X)^3=0$.
The point is that the coordinate transformation $(t,x)\mapsto (T,X)$ with
\[ T=-\frac{t}{t^2-|x|^2},\quad X=\frac{x}{t^2-|x|^2} \]
maps the forward lightcone $\{(t,x): |x|< t, t>0\}$ to the backward lightcone 
$\{(T,X): |X|< -T, T<0\}$ and $t \to \infty$ translates into $T\to 0-$ 
(see Fig.~\ref{fig:Kelvin}). Moreover, 
\[ \tfrac{1}{T^2-|X|^2}v_0 \left (-\tfrac{T}{T^2-|X|^2},
\tfrac{X}{T^2-|X|^2}\right )=\tfrac{\sqrt 2}{(-T)}=:u_0(T,X) \]
and thus,
we are 
led to the study of the stability of the self-similar blowup solution
$u_0$ in the backward lightcone of the origin.
In the context of radial symmetry, this problem was recently addressed by the first author
and Sch\"orkhuber \cite{DonSch12}, see also \cite{Don11, Don12, DonSch12b}
for similar results in the context of wave maps, Yang-Mills equations, and supercritical
wave equations.
However, in the present paper we do not assume any symmetry of the data and hence, we develop
a stability theory similar to \cite{DonSch12} but beyond the radial context.
Furthermore, the instabilities of $u_0$ have a different interpretation in the current setting
and lead to the co-dimension 4 condition in Theorem \ref{thm:main} whereas the blowup
studied in \cite{DonSch12} is stable.
The conformal symmetry, although convenient, does not seem crucial for
our argument. It appears that one can employ similar techniques to study nondispersive solutions
for semilinear wave equations
$(-\partial_t^2+\Delta_x)v(t,x)+v(t,x)|v(t,x)|^{p-1}=0$ with more general $p>3$.

\subsection{Notations}
The arguments for functions defined on Minkowski space are numbered by $0,1,2,3$
and we write $\partial_\mu$, $\mu \in \{0,1,2,3\}$, for the respective derivatives.
Our sign convention for the Minkowski metric $\eta$ is $(-,+,+,+)$.
We use the notation $\partial_y$ for the derivative with respect to the variable $y$.
We employ Einstein's summation convention throughout with latin indices running from $1$
to $3$ and greek indices running from $0$ to $3$, unless otherwise stated. We denote by $\R^+_0$ the set of positive real numbers including $0$.

The letter $C$ (possibly with indices to indicate dependencies) denotes a generic
positive constant which may have a different value at each occurrence.  
The symbol $a\lesssim b$ means $a\leq Cb$ and we abbreviate $a\lesssim b\lesssim a$
by $a\simeq b$. We write $f(x)\sim g(x)$ for $x\to a$ if $\lim_{x\to a}\frac{f(x)}{g(x)}=1$.

For a closed linear operator $\mb L$ on a Banach space we denote its domain by $\mc D(\mb L)$, its spectrum
by $\sigma(\mb L)$, and its point spectrum by $\sigma_p(\mb L)$.
We write $\mb R_{\mb L}(z):=(z-\mb L)^{-1}$ for $z\in \rho(\mb L)=\C
\backslash \sigma(\mb L)$.
The space of bounded operators on a Banach space $\mc X$ is denoted by $\mc B(\mc X)$.

\section{Wave equations and geometry}
\label{sec:geo}
In this section, we present the motivation for using hyperboloidal coordinates in our 
analysis and provide some background. We discuss the main arguments and tools in a 
pedagogical manner to emphasize the relation between spacetime geometry and wave equations 
for readers not familiar with relativistic terminology.
\subsection{Geometric Preliminaries}
A spacetime $(\mathcal{M},g)$ is a four-dimensional paracompact Hausdorff 
manifold $\mathcal{M}$ with a time-oriented Lorentzian metric $g$. 
The cubic wave equation \eqref{eq:main} is posed on the Minkowski 
spacetime $(\mathbb{R}^4,\eta)$. In standard time $t$ and Cartesian coordinates $(x,y,z)$ 
the Minkowski metric reads
\[ \eta = -dt^2 + dx^2 + dy^2 + dz^2\,, \qquad (t,x,y,z)\in \mathbb{R}^4\,.  \]

Minkowski spacetime is \emph{spherically symmetric}, i.e., 
the group $SO(3)$ acts non-trivially by isometry on $(\mathbb{R}^4,\eta)$.
We introduce the quotient space $\mathcal{Q}=\mathbb{R}^4/SO(3)$, and the area 
radius $r:\mathcal{Q}\to\mathbb{R}$ such that the group orbits of points $p\in \mathcal{Q}$ 
have area $4\pi r^2( p )$. The area radius can be written as $r=\sqrt{x^2+y^2+z^2}$ with 
respect to Cartesian coordinates. The flat metric can then be written 
as $\eta = q + r^2 d\sigma^2$, where $q$ is a rank 2 Lorentzian metric 
and $d\sigma^2$ is the standard metric on $S^2$. Choosing the usual angular variables 
for $d\sigma^2$, we obtain the familiar form of the flat spacetime metric in spherical 
coordinates  
\[  \eta = -dt^2 + dr^2 + r^2\left(d\theta^2+\sin^2\theta d\phi^2\right)\,,\quad (t,r,\theta,\phi) \in 
\mathbb{R} \times \mathbb{R}_+ \times [0,\pi] \times [0,2\pi). \]
 A codimension one submanifold is called a \emph{hypersurface} and 
a \emph{foliation} is a one-parameter family of non-intersecting spacelike hypersurfaces. 
A foliation can also be defined by a 
time function from $\mathcal{M}$ to the real line $\mathbb{R}$, whose level sets 
are the hypersurfaces of the foliation. 

We can restrict our discussion of the interaction between hyperbolic equations and spacetime geometry to spherical symmetry without loss of generality because the radial direction is sufficient for exploiting the Lorentzian structure. Working in the two-dimensional quotient spacetime $(\mathcal{Q},q)$ also allows us to illustrate the geometric definitions in two-dimensional plots. 

\subsection{Compactification and Penrose diagrams}

It is useful to introduce Penrose diagrams to depict global features of time foliations in spherically symmetric spacetimes. Penrose presented the construction of the diagrams in his study of the asymptotic behavior of gravitational fields in 1963 \cite{Penrose63}. A beautiful exposition of Penrose diagrams has been given in \cite{Dafermos05}. As we are working in Minkowski spacetime only, the main features of Penrose diagrams of interest to us is the compactification and the preservation of the causal structure. See, for example,
\cite{Christo86, KeelTao} for the application of Penrose compactification to study wave
equations in flat spacetime

The image of the Penrose diagram is a 2-dimensional Minkowski spacetime with a \emph{bounded} global null coordinate system.  Causal concepts extend through the boundary of the map. Consider the rank-2 Minkowski metric $q$ on the quotient manifold $\mathcal{Q}$
\be\label{intro:eta}q=-dt^2+dr^2, \qquad   (t,r) \in \mathbb{R}\times\R^+_0 \,. \ee
To map this metric to a global, bounded, null coordinate system, define $u=t-r$ and  $v=t+r$ for $v \geq u$, and compactify by $U= \arctan u$ and $V=\arctan v$. The quotient metric becomes
\[q=-\frac{1}{\cos^2 V \cos^2 U} dV\,dU, \qquad 
(-\pi/2 < U \leq V < \pi/2). \]
Points at infinity with respect to the original coordinates have finite values with respect to the compactifying coordinates. The singular behavior of the metric in compactifying coordinates at the boundary can be compensated by a conformal rescaling with the conformal factor 
$\Omega= \cos V \cos U$, so that the rescaled metric
\[ \bar{q} = \Omega^2 q= -dU\,dV, \]
is well defined on the domain $(-\pi/2 \leq U \leq V \leq \pi/2)$ including 
points that are at infinity with respect to $q$. 
We say that $q$ can be conformally extended beyond infinity.

\begin{figure}[ht]
\center
\includegraphics[width=0.7\textwidth]{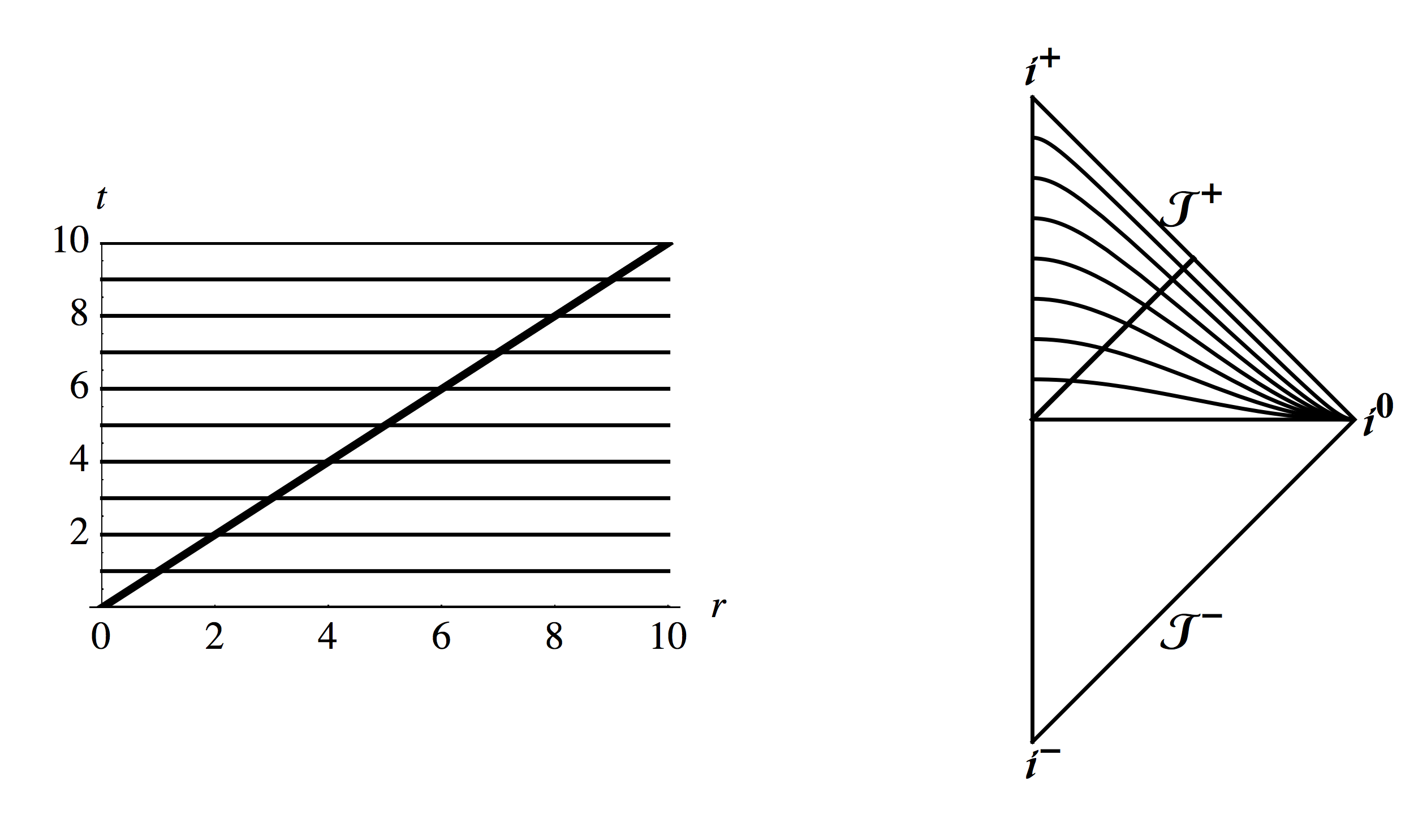}
\caption{The level sets of the standard time $t$ depicted in a spacetime diagram (left panel) and a Penrose diagram (right panel) together with a characteristic line from the origin. The boundary of the Penrose diagram includes the spatial origin and various notions of infinity. Past and future timelike infinity are depicted by points $i^-$ and $i^+$. The vertical line connecting $i^-$ and $i^+$ is the spatial origin $r=0$. Spatial infinity is denoted by the point $i^0$. Null curves reach past and future null infinity, denoted by $\scri^-$ and $\scri^+$, for infinite values of their affine parameter. 
\label{fig:PenroseStandard}}
\end{figure}
 
The Penrose diagram is then drawn using time and space coordinates
\mbox{$T=(V+U)/2$} and $R=(V-U)/2$ (see Fig.~\ref{fig:PenroseStandard}). 
The resulting metric $\bar{q}=-dT^2+dR^2$ is flat. The combined Penrose map 
is given by
\begin{align*} t&\mapsto\frac{1}{2} \left(\tan\left(\frac{T+R}{2}\right) 
+ \tan\left(\frac{T-R}{2}\right)\right), \\ 
 r&\mapsto \frac{1}{2} 
\left(\tan\left(\frac{T+R}{2}\right)-\tan\left(\frac{T-R}{2}\right)\right)\,.
\end{align*}

The boundary $\partial \bar{\mathcal{Q}} = \{T=\pm(\pi-R)$, \mbox{$R\in[0,\pi]\}$} corresponds to points at infinity with respect to the original Minkowski metric.  Asymptotic behavior of fields on $\mathcal{Q}$ can be studied using local differential geometry near this boundary where the conformal factor $\Omega=\cos T+\cos R$ vanishes. The part of the boundary without the points at $R=0,\pi$ is denoted by $\scri=\{T=\pm(\pi-R),R\in(0,\pi)\}$.  This part is referred to as null infinity because null geodesics reach it for an infinite value of their affine parameter. The differential of the conformal factor is non-vanishing at $\scri$, $d\Omega|_{\scri}\ne 0$, and $\scri$ consists of two parts $\scri^-$ and $\scri^+$ referred to as past and future null infinity.   

\subsection{Hyperboloidal coordinates and wave equations}
Equipped with the tools presented in the previous sections we now turn to the interplay 
between wave equations and spacetime geometry. Consider the free wave equation 
\be\label{eq:wave} 
u_{tt}-\Delta u= -\eta^{\mu\nu} \partial_\mu \partial_\nu u = 0\,.
\ee
Radial solutions for the rescaled field $v:=r u$ obey the 2-dimensional free wave equation
\be\label{eq:2d}
v_{tt} - v_{rr} = 0\,,
\ee 
on  $(t,r)\in \R^+_0\times\R^+_0$ with vanishing boundary condition at the
origin.  Initial data are specified on the $t=0$ hypersurface.
The general solution to this system is such that the data propagate to infinity and 
leave nothing behind due to the validity of Huygens' principle. 
Intuitively, this behavior seems to contradict two well-known properties of the 
free wave equation: \emph{conservation of energy} and \emph{time-reversibility}. 

The conserved energy for the free wave equation \eqref{eq:2d} reads
\[ E(v) = \int_0^\infty \frac{1}{2} \left(v_t(t,r)^2 + v_r(t,r)^2 \right) dr\,. \]
The conservation of energy is counterintuitive because the waves propagate to infinity leaving nothing behind. One would expect a natural energy norm to decrease rapidly to zero with a non-positive energy flux at infinity. The conservation of energy, however, implies that at very late times the solution is in some sense similar to the initial state \cite{Tao}. 

Another counterintuitive property of the free wave equation is its 
time-reversibility, meaning that if $u(t,r)$ solves the equation, so does $u(-t,r)$. 
Data on a Cauchy hypersurface determine the solution at all future \emph{and past} times in contrast to parabolic (dissipative) equations which are solvable only forward in time due to loss of energy to the future. 

Both of these counterintuitive properties depend on our description of the problem. We can choose coordinates in which energy conservation and time-reversibility is violated. Of course, it is always possible to find coordinates which break symmetries or hide features of an equation. We argue below that the hyperboloidal coordinates we employ emphasize the intuitive properties of the equation rather than blur them.

The reason behind the conservation of energy integrated along level sets of $t$ can be seen in the Penrose diagram Fig.~\ref{fig:PenroseStandard}. The outgoing characteristic line along which the wave propagates to infinity intersects all leaves of the $t$-foliation. When the energy expression is integrated globally, the energy of the initial wave will therefore still contribute to the result.   The hyperboloidal $T$-foliation depicted in Fig.~\ref{fig:Kelvin}, however, allows for outgoing null rays to leave the leaves of the foliation. Therefore one would expect that the energy flux through infinity is negative when integrated along the leaves of the hyperboloidal foliation.

The wave equation \eqref{eq:2d} has the same form in hyperboloidal coordinates that we employ
\[ w_{TT} - w_{RR} = 0\,, \]
where $w(T,R)=v(-\frac{T}{T^2-R^2},\frac{R}{T^2-R^2})$. The energy conservation and time-reversibility seems valid for this equation as well, but here we have the shrinking, bounded spatial domain $R\in [0,-T)$ where $T\to 0-$. The energy integrated along the leaves of this domain 
\[ E(w) = \int_0^{-T} \frac{1}{2} \left( w_T(T,R)^2 + w_R(T,R)^2 \right) dR\,, \]
decays in time. The energy flux reads
\[ \frac{\partial E}{\partial T} = - \frac{1}{2} \left(w_T(T,-T) - w_R(T,-T)\right)^2 \leq 0\,. \]
The energy flux through infinity vanishes only if the solution is constant or is propagating along future null infinity. When the solution has an outgoing component through future null infinity, the energy decays in time. This behavior is in accordance with physical intuition. 

Consider the time reversibility. The equation in the new coordinates is time-reversible, but the hyperboloidal initial value problem is not. Formally, this is again a consequence of the time dependence of the spatial domain given by $R<-T$. Geometrically, we see in Fig.~\ref{fig:time} that the  union of the past and future domain of dependence of the hyperboloidal surface $T=-1$ covers only a portion of Minkowski spacetime whereas for the Cauchy surface $t=0$ such a union gives the global spacetime.

\begin{figure}[ht]
\center
\includegraphics[width=0.23\textwidth]{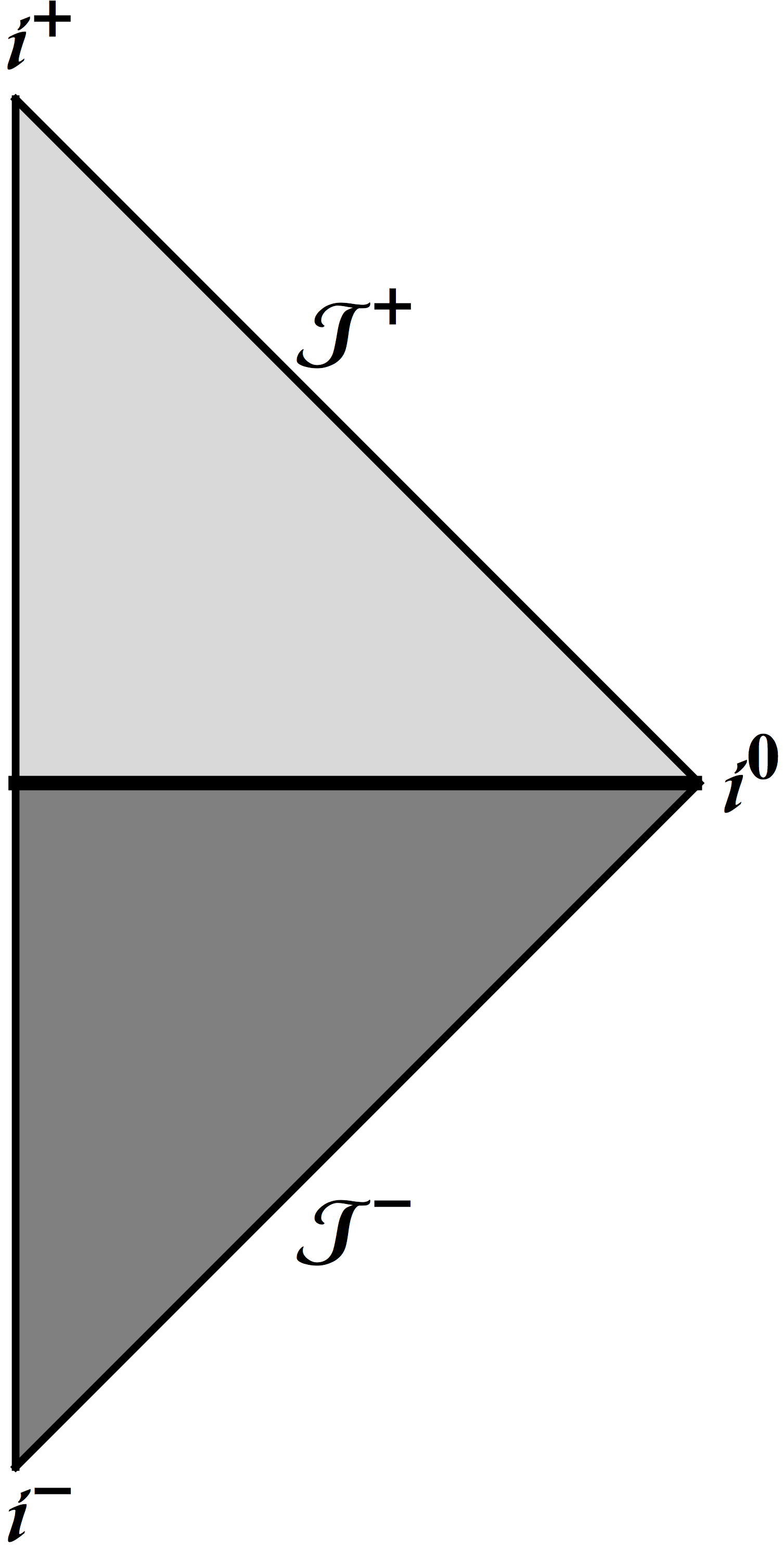} \hspace{2cm}
\includegraphics[width=0.23\textwidth]{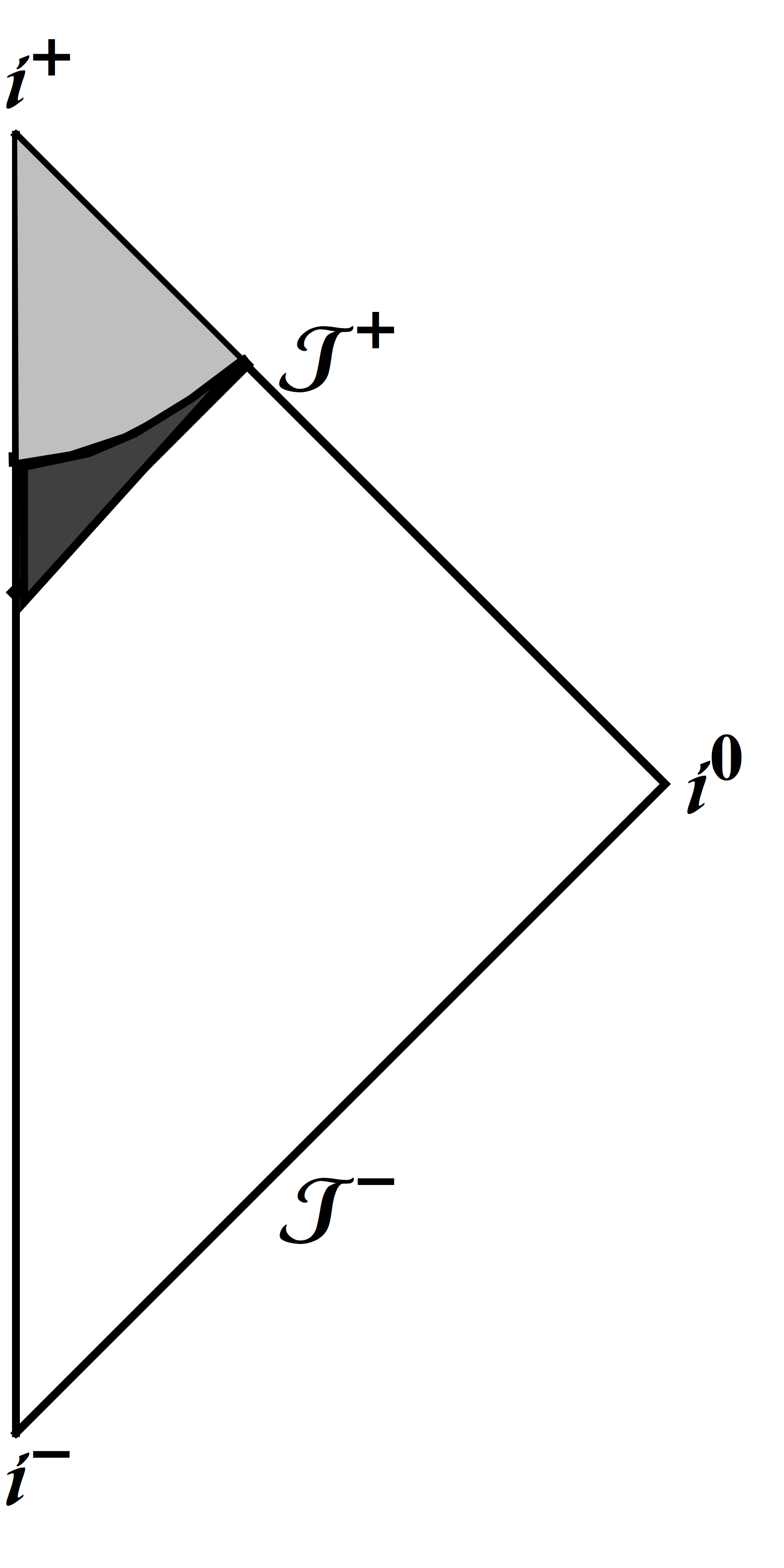}
\caption{Comparison of the future (light gray) and past (dark gray) domains of dependence for the Cauchy surface $t=0$ (left) and the hyperboloidal surface $T=-1$ (right). 
\label{fig:time}}
\end{figure}

In summary, the hyperboloidal foliation given by the Kelvin inversion captures quantitatively the propagation of energy to infinity and leads to a time-irreversible wave propagation problem. Further, the transformation translates asymptotic analysis for $t\to\infty$ to local analysis for $T\to0-$.

\section{Derivation of the equations and preliminaries}

\subsection{First-order formulation and similarity coordinates}
\label{sec:sim}

We start from $(-\partial_T^2+\Delta_X)u(T,X)+u(T,X)^3=0$ in the hyperboloidal coordinates
$T=-\frac{t}{t^2-|x|^2}$, $X=\frac{x}{t^2-|x|^2}$ for the rescaled unknown 
\[ u(T,X)=\frac{1}{T^2-|X|^2}v\left (-\frac{T}{T^2-|X|^2},\frac{X}{T^2-|X|^2}\right ). \]
As discussed in the introduction, the domain we are interested in is $T\in [-1,0)$ and 
$|X|<|T|$.
Our intention is to study the stability of the self-similar solution $u_0(T)=\frac{\sqrt 2}{(-T)}$.
Thus, it is natural to introduce the similarity coordinates
\begin{equation}
\label{eq:sim0}
 \tau=-\log(-T),\quad \xi=\tfrac{X}{(-T)} 
 \end{equation}
with domain $\tau\geq 0$ and $|\xi|<1$.
The derivatives transform according to
\[ \partial_T=e^\tau (\partial_\tau+\xi^j \partial_{\xi^j}), \quad
\partial_{X^j}=e^\tau \partial_{\xi^j}. \]
This implies
\[ \partial_T^2=e^{2\tau}(\partial_\tau^2+\partial_\tau+2\xi^j \partial_{\xi^j}\partial_\tau
+\xi^j \xi^k \partial_{\xi^j}\partial_{\xi^k}+2\xi^j \partial_{\xi^j}) \]
and $\partial_{X_j}\partial_{X^j}=e^{2\tau}\partial_{\xi_j}\partial_{\xi^j}$.
Consequently, for the function 
\[ U(\tau,\xi):=u(-e^{-\tau},e^{-\tau}\xi) \] we obtain
from $(-\partial_T^2+\partial_{X_j}\partial_{X^j})u(T,X)+u(T,X)^3=0$ the equation
\[ [\partial_\tau^2+\partial_\tau+2\xi^j \partial_{\xi^j}\partial_\tau
-(\delta^{jk}-\xi^j \xi^k) \partial_{\xi^j}\partial_{\xi^k}+2\xi^j \partial_{\xi^j}]U(\tau,\xi)
=e^{-2\tau}U(\tau,\xi)^3. \]
In order to get rid of the time-dependent prefactor on the right-hand side, we rescale
and set $U(\tau,\xi)=e^\tau \psi(\tau,\xi)$ which yields
\begin{equation}
\label{eq:psi}
[\partial_\tau^2+3\partial_\tau+2\xi^j \partial_{\xi^j}\partial_\tau
-(\delta^{jk}-\xi^j \xi^k) \partial_{\xi^j}\partial_{\xi^k}+4\xi^j \partial_{\xi^j}+2]\psi(\tau,\xi)
=\psi(\tau,\xi)^3.
\end{equation}
The fundamental self-similar solution is given by 
\[ \psi_0(\tau,\xi):=e^{-\tau}u_0(-e^{-\tau},e^{-\tau}\xi)=\sqrt{2}. \]
Writing $\psi=\sqrt{2}+\phi$ we find the equation
\begin{align}
\label{eq:phi}
[\partial_\tau^2&+3\partial_\tau+2\xi^j \partial_{\xi^j}\partial_\tau
-(\delta^{jk}-\xi^j \xi^k) \partial_{\xi^j}\partial_{\xi^k}+4\xi^j \partial_{\xi^j}+2]\phi(\tau,\xi)
\nonumber \\
&=6\phi(\tau,\xi)+3\sqrt{2}\phi(\tau,\xi)^2+\phi(\tau,\xi)^3.
\end{align}
In summary, we have applied the coordinate transformation
\[ \tau=-\log\left ( \frac{t}{t^2-|x|^2}\right ),\quad \xi=\frac{x}{t} \]
with inverse
\[ t=\frac{e^\tau}{1-|\xi|^2},\quad x=\frac{e^\tau \xi}{1-|\xi|^2} \]
and $\phi(\tau,\xi)$ 
solves Eq.~\eqref{eq:phi} for $\tau > 0$ and $|\xi|<1$
if and only if
\begin{equation}
\label{eq:vphi}
v(t,x)=\tfrac{\sqrt{2}}{t}+\tfrac{1}{t}\phi\left (-\log\tfrac{t}{t^2-|x|^2},\tfrac{x}{t}
\right )
\end{equation}
solves $(-\partial_t^2+\Delta_x)v(t,x)+v(t,x)^3=0$ for $(t,x) \in D^+(\Sigma_{-1})$.

We have $\partial_T u(T,X)=e^{2\tau} (\partial_\tau+\xi^j \partial_{\xi^j}+1)\psi(\tau,\xi)$
and thus, it is natural to use the variables $\phi_1=\phi$, 
$\phi_2=\partial_0 \phi+\xi^j \partial_j \phi+\phi$
in a first-order formulation.
We obtain
\begin{align}
\label{eq:css}
\partial_0 \phi_1&=-\xi^j \partial_j \phi_1-\phi_1+\phi_2 \nonumber \\
\partial_0 \phi_2&=\partial_j \partial^j \phi_1-\xi^j \partial_j \phi_2-2\phi_2
+6\phi_1+3\sqrt{2}\phi_1^2+\phi_1^3.
\end{align}
For later reference we also note that Eq.~\eqref{eq:vphi} implies
\begin{equation}
\label{eq:vtphi}
t^2 \partial_t v(t,x)=-\sqrt 2-\tfrac{2t^2}{t^2-|x|^2}\left (\tfrac{x^j}{t}\partial_j
\phi_1+\phi_1\right )+\tfrac{t^2+|x|^2}{t^2-|x|^2}\phi_2
\end{equation}
where it is understood, of course, that $\phi_1(\tau,\xi)$ and $\phi_2(\tau,\xi)$ are evaluated
at $\tau=-\log\frac{t}{t^2-|x|^2}$ and $\xi=\frac{x}{t}$.

\subsection{Norms}

Since our approach is perturbative in nature, the function space in which we study
Eq.~\eqref{eq:css} should be determined by the \emph{free} version
of Eq.~\eqref{eq:css}, i.e., 
\begin{align*}
\partial_0 \phi_1&=-\xi^j \partial_j \phi_1-\phi_1+\phi_2 \nonumber \\
\partial_0 \phi_2&=\partial_j \partial^j \phi_1-\xi^j \partial_j \phi_2-2\phi_2.
\end{align*} 
The natural choice for a norm is derived from the 
standard energy $\dot H^1 \times L^2$ of the free
wave equation.
In the present formulation this translates into
\[ \|\phi_1(\tau,\cdot)\|_{\dot H^1(B)}+\|\phi_2(\tau,\cdot)\|_{L^2(B)} \]
where $B=\{\xi \in \R^3: |\xi|<1\}$.
However, there is a slight technical problem since this is only a seminorm (the point
is that we are working on the bounded domain $B$).
In order to go around this difficulty, let us for the moment return to the radial context
and consider the free wave equation in $\R^{1+3}$,
\[ u_{tt}-u_{rr}-\tfrac{2}{r}u_r=0, \]
in the standard coordinates $t$ and $r=|x|$.
Now we make the following observation.
The conserved energy is given by
\[ E(u)=\tfrac12 \int_0^\infty [u_t^2+u_r^2]r^2 dr. \]
On the other hand, by setting $v=ru$, we obtain
\[ v_{tt}-v_{rr}=0 \]
with conserved energy $\frac12 \int_0^\infty [v_t^2+v_r^2] dr$, or, in terms of $u$,
\[ E'(u)=\tfrac12 \int_0^\infty [r^2 u_t^2+(ru_r+u)^2]dr. \]
The obvious question now is: how are $E$ and $E'$ related?
An integration by parts shows that $E$ and $E'$ are equivalent, up to a boundary term
$\lim_{r\to \infty}r u(r)^2$ which may be ignored by assuming some decay at spatial
infinity.
However, if we consider the \emph{local} energy contained in a ball of radius $R$, 
the boundary term can no longer be ignored
and one has the identity
\[ E_R'(u):=\tfrac12 \int_0^R [r^2 u_t^2+(ru_r+u)^2]dr=\tfrac12 R u(R)^2+\tfrac12 \int_0^R [u_t^2+u_r^2]r^2 dr. \]
The expression on the right-hand side is the standard energy with the term $\frac12 Ru(R)^2$ added.
This small modification has important consequences because unlike the standard energy, 
this now defines a norm.
Furthermore, $E'_R(u)$ is bounded along the wave flow since it is the local version of
a positive definite conserved quantity.

In the nonradial context the above discussion suggests to take
\[ \|\phi_1(\tau,\cdot)\|_{\dot H^1(B)}+\|\phi_1(\tau,\cdot)\|_{L^2(\partial B)}
+\|\phi_2(\tau,\cdot)\|_{L^2(B)}. \]
This norm is not very handy but fortunately, we have equivalence to $H^1\times L^2(B)$ as the
following result shows. 

\begin{lemma}
\label{lem:equivH1}
We have \footnote{As usual, $\|f\|_{L^2(\partial B)}$ has to be understood in the trace sense.}
\[ \|f\|_{H^1(B)}\simeq \|f\|_{\dot H^1(B)}+\|f\|_{L^2(\partial B)}. \]
\end{lemma}

\begin{proof}
For $x \in \R^3$ we write $r=|x|$ and $\omega=\frac{x}{|x|}$.
With this notation we have $f(x)=f(r\omega)$ and
\[ \|f\|_{L^2(B)}^2=\int_0^1 \int_{\partial B} |f(r\omega)|^2 d\sigma(\omega)r^2dr \]
where $d\sigma$ denotes the surface measure on the sphere.
First, we prove $\|f\|_{L^2(B)}\lesssim \|f\|_{\dot H^1(B)}+\|f\|_{L^2(\partial B)}$.
By density it suffices to consider $f\in C^\infty(\overline B)$.
The fundamental theorem of calculus and Cauchy-Schwarz imply
\begin{align*}
r f(r\omega)=\int_0^r \partial_s [sf(s\omega)]ds
\leq \left (\int_0^1 |\partial_r [rf(r\omega)]|^2 dr \right )^{1/2}.
\end{align*}
Expanding the square and integration by parts yield
\begin{align*} \int_0^1 |\partial_r [rf(r\omega)]|^2 dr&=\int_0^1 |\partial_r f(r\omega)|^2 r^2 dr
+\int_0^1 r\partial_r |f(r\omega)|^2 dr \\
&\quad +\int_0^1 |f(r\omega)|^2 dr \\
&=|f(\omega)|^2 + \int_0^1 |\partial_r f(r\omega)|^2 r^2 dr
\end{align*}
and thus,
\[ r^2 |f(r\omega)|^2\leq |f(\omega)|^2 + \int_0^1 |\omega^j \partial_j f(r\omega)|^2 r^2 dr. \]
Integrating this inequality over the ball $B$ yields the desired estimate.
In order to finish the proof, it suffices to show that $\|f\|_{L^2(\partial B)}\lesssim
\|f\|_{H^1(B)}$, but this is just the trace theorem (see e.g.~\cite{Eva98}, p.~258, Theorem 1).
\end{proof}

\section{Linear perturbation theory}
\label{sec:lin}

The goal of this section is to develop a functional analytic framework for studying the Cauchy
problem for the \emph{linearized} equation
\begin{align}
\label{eq:lincss0}
\partial_0 \phi_1&=-\xi^j \partial_j \phi_1-\phi_1+\phi_2 \nonumber \\
\partial_0 \phi_2&=\partial_j \partial^j \phi_1-\xi^j \partial_j \phi_2-2\phi_2
+6\phi_1.
\end{align}
The main difficulty lies with the fact that the differential operators involved are not
self-adjoint. It is thus natural to apply semigroup theory for studying
Eq.~\eqref{eq:lincss0}.
Before doing so, however, we commence with a heuristic discussion on instabilities.
The equation $(-\partial_T^2+\Delta_X)u(T,X)+u(T,X)^3=0$ is invariant
under time translations $T\mapsto T-a$ and the three 
Lorentz boosts for each direction $X^j$
\[ \left \{ \begin{array}{rcl}
T &\mapsto &T\cosh a-X^j \sinh a \\
X^j &\mapsto &-T\sinh a+X^j \cosh a \\
X^k &\mapsto &X^k \quad (k\not=j) \end{array} \right. \]
where $a\in\R$ is a parameter (the rapidity in case of the Lorentz boost).
In general, if $u_a$ is a one-parameter family of solutions to a nonlinear equation
$F(u_a)=0$, one obtains (at least formally)
\[ 0=\partial_a F(u_a)=DF(u_a)\partial_a u_a \]
and thus, $\partial_a u_a$ is a solution of the linearization of $F(u)=0$ at $u=u_a$.
In our case we linearize around the solution $u_0(T,X)=\frac{\sqrt{2}}{(-T)}$.
The time translation symmetry yields the one-parameter family 
$u_a(T,X):=\frac{\sqrt{2}}{a-T}$ and we have $\partial_a u_a(T,X)|_{a=0}=-\frac{\sqrt{2}}{T^2}$.
Taking into account the above transformations that led from $u$ to $\phi_1$, $\phi_2$, we 
obtain (after a suitable normalization) the functions
\begin{align}
\label{eq:mode1}
\phi_1(\tau,\xi)=e^\tau,\quad
\phi_2(\tau,\xi)=2e^\tau
\end{align}
and a simple calculation shows that \eqref{eq:mode1} indeed solve Eq.~\eqref{eq:lincss0}.
Thus, there exists a growing solution of Eq.~\eqref{eq:lincss0}.
Similarly, for the Lorentz boosts we consider 
\[ u_{a,j}(T,X)=\frac{\sqrt{2}}{X^j\sinh a-T\cosh a} \] and thus, 
$\partial_a u_{a,j}(T,X)|_{a=0}=-\frac{\sqrt{2}}{T^2}X^j$.
By recalling that $\frac{X^j}{(-T)}=\xi^j$, this yields the functions
\begin{align} 
\label{eq:mode0}
\phi_1(\tau,\xi)=\xi^j, \quad
\phi_2(\tau,\xi)=2\xi^j 
\end{align}
and it is straightforward to check that \eqref{eq:mode0} indeed solve Eq.~\eqref{eq:lincss0}.
This time the solution \eqref{eq:mode0} is not growing in $\tau$ 
but it is not decaying either.
It is important to emphasize that in our context, 
the time translation symmetry leads to a \emph{real} instability.
The reason being that $u_a(T,X)=\frac{\sqrt{2}}{a-T}$ yields the solution
\[ v_a(t,x)=\frac{1}{t^2-|x|^2}\frac{\sqrt{2}}{a+\frac{t}{t^2-|x|^2}}=\frac{\sqrt{2}}{t+a(t^2-|x|^2)} \]
of the original problem. This solution is part of a two-parameter family conjectured to
describe generic radial solutions of the focusing cubic wave equation \cite{BizZen09}.
If $a\not=0$, $v_a(t,x)$ decays like $t^{-2}$ as $t\to\infty$
for each fixed $x\in \R^3$.
This is the generic (dispersive) decay. 
On the other hand, the Lorentz transforms lead to apparent instabilities since
the function 
$u_{a,j}$ yields the solution
$v_{a,j}(t,x)=\frac{\sqrt 2}{t\cosh a+x^j\sinh a}$ of the original problem which still 
displays the nondispersive decay. 
Consequently, we expect a co-dimension one manifold of initial data that lead to 
nondispersive decay, 
as mentioned in the introduction. Since we are working with a fixed $u_0$, however, there is 
a four-dimensional unstable subspace of the linearized operator (to be defined below). 
This observation eventually leads to the co-dimension 4 statement in our Theorem \ref{thm:main}.
Note that other symmetries of the equation such as scaling, space translations, and
space rotations do not play a role in this context as the solution $u_0$ is invariant
under these. 

\subsection{A semigroup formulation for the free evolution}
We start the rigorous treatment by considering the free wave equation 
in similarity coordinates given by the system
\begin{align}
\label{eq:freecss}
\partial_0 \phi_1&=-\xi^j \partial_j \phi_1-\phi_1+\phi_2 \nonumber \\
\partial_0 \phi_2&=\partial_j \partial^j \phi_1-\xi^j \partial_j \phi_2-2\phi_2.
\end{align}
From Eq.~\eqref{eq:freecss} we read off the generator
\[ \tilde{\mb L}_0 \mb u(\xi)=\left (\begin{array}{c}
-\xi^j \partial_j u_1(\xi)-u_1(\xi)+u_2(\xi) \\
\partial_j \partial^j u_1(\xi)-\xi^j \partial_j u_2(\xi)-2 u_2(\xi) \end{array} \right ), \]
acting on functions in $\mc D(\tilde {\mb L}_0):=H^2(B)\cap C^2(\overline B
\backslash\{0\})\times H^1(B)\cap C^1(\overline B\backslash \{0\})$. 
With this notation we rewrite Eq.~\eqref{eq:freecss} as an ODE,
\[ \tfrac{d}{d\tau}\Phi(\tau)=\tilde{\mb L}_0 \Phi(\tau). \]
The appropriate framework for studying such a problem is provided by semigroup theory, i.e.,
our goal is to find a suitable Hilbert space $\mc H$ such that
there exists a map $\mb S_0: [0,\infty)\to \mc B(\mc H)$ satisfying
\begin{itemize}
\item $\mb S_0(0)=\mathrm{id}_\mc H$,
\item $\mb S_0(\tau)\mb S_0(\sigma)=\mb S_0(\tau+\sigma)$ for all $\tau,\sigma \geq 0$,
\item $\lim_{\tau\to 0+}\mb S_0(\tau)\mb u=\mb u$ for all $\mb u\in \mc H$,
\item $\lim_{\tau\to 0+}\frac{1}{\tau}[\mb S_0(\tau)\mb u-\mb u]=\mb L_0 \mb u$
for all $\mb u\in \mc D(\mb L_0)$ where $\mb L_0$ is the closure of $\tilde{\mb L}_0$.
\end{itemize}
Given such an $\mb S_0$, the function $\Phi(\tau)=\mb S_0(\tau)
\Phi(0)$ solves $\frac{d}{d\tau}\Phi(\tau)=\mb L_0\Phi(\tau)$.

Motivated by the above discussion we define a sesquilinear form on 
$\tilde{\mc H}:=H^1(B)\cap C^1(B)\times L^2(B)\cap C(B)$
by
\[ (\mb u|\mb v):=\int_B \partial_j u_1(\xi)\overline{\partial^j v_1(\xi)}d\xi
+\int_{\partial B}
u_1(\omega)\overline{v_1(\omega)}d\sigma(\omega)
+\int_B u_2(\xi)\overline{v_2(\xi)}d\xi. \]
Lemma \ref{lem:equivH1} implies that $(\cdot|\cdot)$ is an inner product
on $\tilde{\mc H}$ and as usual, we denote the induced norm by $\|\cdot\|$.
Furthermore, we write $\mc H$ for the completion of $\tilde{\mc H}$ with respect to
$\|\cdot\|$.
We remark that $\mc H$ is equivalent to $H^1(B)\times L^2(B)$ as a Banach space by
Lemma \ref{lem:equivH1}.

\begin{proposition}
\label{prop:free}
The operator $\tilde{\mb L}_0: \mc D(\tilde{\mb L}_0)\subset \mc H\to \mc H$ is closable 
and its closure, denoted by $\mb L_0$,
generates a strongly continuous semigroup $\mb S_0: [0,\infty) \to \mc B(\mc H)$ satisfying
$\|\mb S_0(\tau)\|\leq e^{-\frac12 \tau}$
for all $\tau\geq 0$.
In particular, we have 
$\sigma(\mb L_0)\subset \{z\in \C: \Re z\leq -\tfrac12\}$.
\end{proposition}

The proof of Proposition \ref{prop:free} requires the following technical lemma.

\begin{lemma}
\label{lem:degell}
Let $f\in L^2(B)$ and $\varepsilon>0$ be arbitrary.
Then there exists a function $u \in H^2(B)\cap C^2(\overline B\backslash \{0\})$ such that 
$g \in L^2(B)\cap C(\overline B\backslash \{0\})$, defined by
\begin{equation}
\label{eq:degell}
 g(\xi):=-(\delta^{jk}-\xi^j \xi^k)\partial_j \partial_k u(\xi)+5\xi^j\partial_j u(\xi)
+\tfrac{15}{4}u(\xi), 
\end{equation}
satisfies $\|f-g\|_{L^2(B)}<\varepsilon$.
\end{lemma}

\begin{proof}
Since $C^\infty(\overline B)\subset L^2(B)$ is dense, 
we can find a $\tilde g\in C^\infty(\overline B)$
such that $\|f-\tilde g\|_{L^2(B)}<\frac{\varepsilon}{2}$.
We consider the equation
\begin{equation}
\label{eq:degellgtilde}
-(\delta^{jk}-\xi^j \xi^k)\partial_j \partial_k u(\xi)+5\xi^j\partial_j u(\xi)
+\tfrac{15}{4}u(\xi)=\tilde g(\xi).
\end{equation} 
In order to solve Eq.~\eqref{eq:degellgtilde} we define
$\rho(\xi)=|\xi|$, $\omega(\xi)=\frac{\xi}{|\xi|}$
and note that 
\[ \partial_j\rho(\xi)=\omega_j(\xi),\quad
\partial_j \omega^k(\xi)
=\frac{\delta_j{}^k-\omega_j(\xi)\omega^k(\xi)}{\rho(\xi)}. \]
Thus, interpreting $\rho$ and $\omega$ as new coordinates, we obtain
\begin{align*} 
\xi^j \partial_j u(\xi)&=\rho \partial_\rho u(\rho \omega) \\
\xi^j \xi^k \partial_j \partial_k u(\xi)&=\xi^j \partial_{\xi^j} [\xi^k \partial_{\xi^k} u(\xi)]
-\xi^j \partial_j u(\xi)=\rho^2 \partial_\rho^2 u(\rho \omega)
\end{align*}
as well as
\[ \partial^j \partial_j u(\rho \omega)=\left [\partial_\rho^2+\frac{d-1}{\rho}\partial_\rho
+\frac{\delta^{jk}-\omega^j \omega^k}{\rho^2}\partial_{\omega^j}\partial_{\omega^k}
-\frac{d-1}{\rho^2}\omega^j\partial_{\omega^j}\right ]u(\rho \omega) \]
where $d=3$ is the spatial dimension.
Consequently, Eq.~\eqref{eq:degellgtilde} can be written as
\begin{equation}
\label{eq:degellrw} 
\left [-(1-\rho^2)\partial_\rho^2-\tfrac{2}{\rho}\partial_\rho+5\rho \partial_\rho
+\tfrac{15}{4}-\tfrac{1}{\rho^2}\Delta_{S^2} \right ]u(\rho \omega)=\tilde g(\rho \omega) 
\end{equation}
where $-\Delta_{S^2}$ is the Laplace-Beltrami operator on $S^2$.
The operator $-\Delta_{S^2}$ is self-adjoint on $L^2(S^2)$ and we have
 $\sigma(-\Delta_{S^2})=\sigma_p(-\Delta_{S^2})=\{\ell(\ell+1): \ell\in \N_0\}$.
The eigenspace to the eigenvalue $\ell(\ell+1)$ is $(2\ell+1)$-dimensional and spanned
by the spherical harmonics $\{Y_{\ell,m}: m \in \Z, -\ell \leq m \leq \ell\}$ which 
are obtained by restricting harmonic homogeneous
polynomials in $\R^3$ to the two-sphere $S^2$, see e.g.~\cite{AtkHan12} for an up-to-date
account of this classical subject.
We may expand $\tilde g$ according to
\[ \tilde g(\rho \omega)=\sum_{\ell,m}^\infty g_{\ell,m}(\rho)Y_{\ell,m}(\omega) \]
where $\sum_{\ell,m}^\infty$ is shorthand for $\sum_{\ell=0}^\infty \sum_{m=-\ell}^\ell$
and for any fixed $\rho \in [0,1]$, the sum converges in $L^2(S^2)$, see \cite{AtkHan12}, 
p.~66, Theorem 2.34.
The expansion coefficient $g_{\ell,m}(\rho)$ is given by
\[ g_{\ell,m}(\rho)=(\tilde g(\rho\: \cdot)|Y_{\ell,m})_{L^2(S^2)}
:=\int_{S^2} \tilde g(\rho \omega)\overline{Y_{\ell,m}(\omega)}d\sigma(\omega) \]
and by dominated convergence it follows that $g_{\ell,m}\in C^\infty[0,1]$.
Furthermore, by using the identity $Y_{\ell,m}=[\ell(\ell+1)]^{-1}(-\Delta_{S^2})Y_{\ell,m}$
and the self-adjointness of $-\Delta_{S^2}$ on $L^2(S^2)$, we obtain
\begin{align*} 
g_{\ell,m}(\rho)&=\tfrac{1}{\ell(\ell+1)}\big(\tilde g(\rho\:\cdot)\big |(-\Delta_{S^2}) 
Y_{\ell,m}\big)_{L^2(S^2)} \\
&=\tfrac{1}{\ell(\ell+1)}\big ((-\Delta_{S^2})\tilde g(\rho\:\cdot) \big |Y_{\ell,m} \big)_{L^2(S^2)}. 
\end{align*}
Consequently, by iterating this argument
we see that the smoothness of $\tilde g$ implies the pointwise decay 
$\|g_{\ell,m}\|_{L^\infty(0,1)}\leq C_M \ell^{-M}$ for any $M \in \N$ and all $\ell \in \N$.
Now we set
\[ g_N(\xi):=\sum_{\ell,m}^N g_{\ell,m}(|\xi|)Y_{\ell,m}(\tfrac{\xi}{|\xi|}) \]
and note that $\|g_N(\rho \:\cdot)-\tilde g(\rho\:\cdot)\|_{L^2(S^2)}\to 0$ as $N\to\infty$.
Furthermore, by $\|g_{\ell,m}\|_{L^\infty(0,1)}\lesssim \ell^{-2}$ for all $\ell \in \N$ we infer 
\[ \sup_{\rho \in (0,1)}\|g_N(\rho\:\cdot)-\tilde g(\rho\:\cdot)\|_{L^2(S^2)}\lesssim 1 \] 
for all $N\in \N$ and dominated convergence yields
\[ \|g_N-\tilde g\|_{L^2(B)}^2=\int_0^1 \|g_N(\rho\:\cdot)-\tilde g(\rho\:\cdot)\|_{L^2(S^2)}^2
\rho^2 d\rho \to 0 \]
as $N\to\infty$.
Thus, we may choose $N$ so large that $\|g_N-\tilde g\|_{L^2(B)}<\frac{\varepsilon}{2}$.

By making the ansatz $u(\rho \omega)=\sum_{\ell,m}^N u_{\ell,m}(\rho)
Y_{\ell,m}(\omega)$ we derive from Eq.~\eqref{eq:degellrw} the (decoupled) 
system
\begin{align}
\label{eq:degellsep}
\Big [-(1-\rho^2)\partial_\rho^2 -\tfrac{2}{\rho}\partial_\rho +5\rho \partial_\rho
+ \tfrac{15}{4}+\tfrac{\ell(\ell+1)}{\rho^2} \Big ]u_{\ell,m}(\rho)=g_{\ell,m}(\rho)
\end{align}
for $\ell\in \N_0$, $\ell \leq N$, and $-\ell \leq m \leq \ell$.
Eq.~\eqref{eq:degellsep} has regular singular points at $\rho=0$
and $\rho=1$ with Frobenius indices $\{\ell,-\ell-1\}$ and $\{0,-\frac12\}$, respectively.
In fact, solutions to Eq.~\eqref{eq:degellsep} can be given in terms of hypergeometric functions.
In order to see this, define a new variable $v_{\ell,m}$ by $u_{\ell,m}(\rho)=\rho^\ell
v_{\ell,m}(\rho^2)$.
Then, Eq.~\eqref{eq:degellsep} with $g_{\ell,m}=0$ is equivalent to
\begin{equation}
\label{eq:hypgeo}
 z(1-z)v_{\ell,m}''(z)+[c-(a+b+1)z]v_{\ell,m}'(z)-ab v_{\ell,m}(z)=0 
 \end{equation}
with $a=\frac12(\frac32+\ell)$, $b=a+\frac12$, $c=\frac32+\ell$, and
$z=\rho^2$.
We immediately obtain the two solutions
\begin{align*} 
\phi_{0,\ell}(z)&={}_2F_1(\tfrac{3+2\ell}{4},\tfrac{5+2\ell}{4},\tfrac{3+2\ell}{2}; z) \\
\phi_{1,\ell}(z)&={}_2F_1(\tfrac{3+2\ell}{4},\tfrac{5+2\ell}{4}, \tfrac32; 1-z)
\end{align*}
where ${}_2F_1$ is the standard hypergeometric function, cf.~\cite{DLMF, Kri10}.
For later reference we also state a third solution, $\tilde \phi_{1,\ell}$, given by
\begin{equation}
\label{eq:tildephi}
\tilde \phi_{1,\ell}(z)=(1-z)^{-\frac12}{}_2F_1(\tfrac{3+2\ell}{4}, \tfrac{1+2\ell}{4},
\tfrac12; 1-z).
\end{equation}
Note that $\phi_{0,\ell}$ is analytic around $z=0$ whereas $\phi_{1,\ell}$
is analytic around $z=1$.
As a matter of fact, $\phi_{1,\ell}$ can be represented in terms of elementary functions
and we have
\begin{equation}
\label{eq:phi1elem}
 \phi_{1,\ell}(z)=\frac{1}{(2\ell+1)\sqrt{1-z}}\left [
\left (1-\sqrt{1-z}\right )^{-\ell-\frac12}-\left (1+\sqrt{1-z} \right)^{-\ell-\frac12} \right ], 
\end{equation}
see \cite{DLMF}. This immediately shows that
$|\phi_{1,\ell}(z)|\to \infty$ as $z\to 0+$ which implies that
$\phi_{0,\ell}$ and $\phi_{1,\ell}$ are linearly independent.
Transforming back, we obtain the two solutions
$\psi_{j,\ell}(\rho)=\rho^\ell \phi_{j,\ell}(\rho^2)$, $j=0,1$,
of Eq.~\eqref{eq:degellsep} with $g_{\ell,m}=0$.
By differentiating the Wronskian $W(\psi_{0,\ell},\psi_{1,\ell})=\psi_{0,\ell}\psi_{1,\ell}'
-\psi_{0,\ell}'\psi_{1,\ell}$ and inserting the equation, we infer
\[ W(\psi_{0,\ell},\psi_{1,\ell})'(\rho)=\left (\tfrac{3\rho}{1-\rho^2}-\tfrac{2}{\rho}\right )
W(\psi_{0,\ell},\psi_{1,\ell})(\rho) \]
which implies
\begin{equation}
\label{eq:Wcl} 
W(\psi_{0,\ell},\psi_{1,\ell})(\rho)=\frac{c_\ell}{\rho^2 (1-\rho^2)^\frac32} 
\end{equation}
for some constant $c_\ell$.
In order to determine the precise value of $c_\ell$, we first note that
\[
\psi_{j,\ell}'(\rho)=2\rho^{\ell+1}\phi_{j,\ell}'(\rho^2)+\ell \rho^{\ell-1}\phi_{j,\ell}(\rho^2).
\]
For the following we recall the differentiation formula \cite{DLMF}
\begin{equation}
\label{eq:diffhypgeo} 
\tfrac{d}{dz}{}_2F_1(a,b,c;z)=\tfrac{ab}{c}{}_2F_1(a+1,b+1,c+1;z)
\end{equation}
which is a direct consequence of the series representation of the hypergeometric
function.
Furthermore, by the formula \cite{DLMF}
\begin{equation}
\label{eq:hypgeoat1} 
\lim_{z \to 1-}[(1-z)^{a+b-c}{}_2F_1(a,b,c;z)]=\frac{\Gamma(c)\Gamma(a+b-c)}{\Gamma(a)
\Gamma(b)},
\end{equation}
valid for $\Re(a+b-c)>0$, we obtain
\begin{equation} 
\label{eq:phi0at1}
\lim_{z\to 1-}[(1-z)^\frac12 \phi_{0,\ell}(z)]=
\frac{\Gamma(\frac{3+2\ell}{2})\Gamma(\frac12)}{\Gamma(\frac{3+2\ell}{4})
\Gamma(\frac{5+2\ell}{4})}=2^{\ell+\frac12} 
\end{equation}
as well as
\[ \lim_{z\to 1-}[(1-z)^\frac32 \phi_{0,\ell}'(z)]=\frac{\Gamma(\frac{3+2\ell}{2})
\Gamma(\frac32)}{\Gamma(\frac{3+2\ell}{4})\Gamma(\frac{5+2\ell}{4})}=2^{\ell-\frac12}
\]
where we used the identity $\Gamma(x)\Gamma(x+\frac12)=\pi^\frac12 2^{1-2x}\Gamma(2x)$.
This yields
\begin{align*}
c_\ell&=\rho^2(1-\rho^2)^\frac32 W(\psi_{0,\ell},\psi_{1,\ell})(\rho) \\
&=\rho^2(1-\rho^2)^\frac32 \rho^\ell \phi_{0,\ell}(\rho^2)[2\rho^{\ell+1}\phi_{1,\ell}'(\rho^2)
+\ell \rho^{\ell-1}\phi_{1,\ell}(\rho^2)] \\
&\quad -\rho^2(1-\rho^2)^\frac32 \rho^\ell \phi_{1,\ell}(\rho^2)
[2\rho^{\ell+1}\phi_{0,\ell}'(\rho^2)
+\ell \rho^{\ell-1}\phi_{0,\ell}(\rho^2)] \\
&=-2\lim_{\rho\to 1-}(1-\rho^2)^{\frac32}\phi_{0,\ell}'(\rho^2) \\
&=-2^{\ell+\frac12}.
\end{align*}
By the variation of constants formula, 
a solution to Eq.~\eqref{eq:degellsep} is given by
\begin{align} 
\label{eq:ulm}
u_{\ell,m}(\rho)=&-\psi_{0,\ell}(\rho)\int_\rho^1 
\frac{\psi_{1,\ell}(s)}{W(\psi_{0,\ell},\psi_{1,\ell})(s)} \frac{g_{\ell,m}(s)}{1-s^2}ds \nonumber \\
&-\psi_{1,\ell}(\rho)\int_0^\rho
\frac{\psi_{0,\ell}(s)}{W(\psi_{0,\ell},\psi_{1,\ell})(s)} \frac{g_{\ell,m}(s)}{1-s^2}ds.
\end{align}

We claim that $u_{\ell,m} \in C^2(0,1]$.
By formally differentiating Eq.~\eqref{eq:ulm} we find
\[ u_{\ell,m}''(\rho)=-\frac{g_{\ell,m}(\rho)}{1-\rho^2}-\psi_{0,\ell}''(\rho)I_{1,\ell}(\rho)
-\psi_{1,\ell}''(\rho)I_{0,\ell}(\rho) \]
where $I_{j,\ell}$, $j=0,1$, denote the respective integrals in Eq.~\eqref{eq:ulm}.
This implies $u_{\ell,m}\in C^2(0,1)$ but $u_{\ell,m}''(\rho)$ has an apparent singularity
at $\rho=1$.
We have the asymptotics $\psi_{0,\ell}''(\rho)I_{1,\ell}(\rho)\simeq (1-\rho)^{-1}$ and
$\psi_{1,\ell}''(\rho)I_{0,\ell}(\rho)\simeq 1$ as $\rho \to 1-$.
Thus, a necessary condition for $\lim_{\rho \to 1-}u_{\ell,m}''(\rho)$ to exist is
\[ a_{\ell,m}:=\lim_{\rho \to 1-}[(1-\rho^2)\psi_{0,\ell}''(\rho)I_{1,\ell}(\rho)]=-g_{\ell,m}(1). \]
This limit can be computed by de l'H\^opital's rule, i.e., we write
\begin{align*} 
a_{\ell,m}&=\lim_{\rho \to 1-}\frac{I_{1,\ell}(\rho)}{[(1-\rho^2)\psi_{0,\ell}''(\rho)]^{-1}} \\
&=\lim_{\rho \to 1-}\frac{I_{1,\ell}'(\rho)}{-[(1-\rho^2)\psi_{0,\ell}''(\rho)]^{-2}
[(1-\rho^2)\psi_{0,\ell}^{(3)}(\rho)-2\rho\psi_{0,\ell}''(\rho)]}.
\end{align*}
We have
\[ \lim_{\rho\to 1-}[(1-\rho^2)^{-\frac12}I_{1,\ell}'(\rho)]=-\tfrac{1}{c_\ell}
\lim_{\rho\to 1-}[\rho^2\psi_{1,\ell}(\rho)g_{\ell,m}(\rho)]=-\frac{g_{\ell,m}(1)}{c_\ell} \]
and thus, it suffices to show that
\begin{align}
\label{eq:cell} 
-\frac{1}{c_\ell}&=\lim_{\rho\to 1-}\frac{(1-\rho^2)\psi_{0,\ell}^{(3)}(\rho)-2\rho \psi_{0,\ell}''(\rho)}
{(1-\rho^2)^\frac12 [(1-\rho^2)\psi_{0,\ell}''(\rho)]^2} \nonumber \\
&=\lim_{\rho\to 1-}\frac{(1-\rho^2)^\frac72 \psi_{0,\ell}^{(3)}(\rho)-2\rho(1-\rho^2)^\frac52 
\psi_{0,\ell}''(\rho)}
{[(1-\rho^2)^\frac52 \psi_{0,\ell}''(\rho)]^2}.
\end{align}
Note that
\begin{align*}
\psi_{0,\ell}''(\rho)&=4\rho^{\ell+2}\phi_{0,\ell}''(\rho^2)+\mbox{lower order derivatives} \\
\psi_{0,\ell}^{(3)}(\rho)&=8\rho^{\ell+3}\phi_{0,\ell}^{(3)}(\rho^2)+\mbox{lower order derivatives}.
\end{align*}
Consequently, from the definition of $\phi_{0,\ell}$ and Eqs.~\eqref{eq:diffhypgeo},
\eqref{eq:hypgeoat1} we infer
\begin{align*}
\lim_{\rho\to 1-}[(1-\rho^2)^\frac52 \psi_{0,\ell}''(\rho)]&=4\lim_{z\to 1-}
[(1-z)^\frac52\phi_{0,\ell}''(z)]=-3c_\ell \\
\lim_{\rho\to 1-}[(1-\rho^2)^\frac72 \psi_{0,\ell}^{(3)}(\rho)]&=8\lim_{z\to 1-}
[(1-z)^\frac72\phi_{0,\ell}^{(3)}(z)]=-15 c_\ell 
\end{align*}
which proves \eqref{eq:cell}.
We have $g_{\ell,m}\in 
C^\infty[0,1]$ and thus, in order to prove the claim $u_{\ell,m}\in C^2(0,1]$, it
suffices to show that $\rho \mapsto (1-\rho^2)\psi_{0,\ell}''(\rho)I_{1,\ell}(\rho)$
belongs to $C^1(0,1]$. We write the integrand in $I_{1,\ell}$ as
\[ \frac{\psi_{1,\ell}(s)}{W(\psi_{0,\ell},\psi_{1,\ell})(s)}\frac{g_{\ell,m}(s)}{1-s^2}
=(1-s)^\frac12 O(1) \]
where in the following, $O(1)$ stands for a suitable function in $C^\infty(0,1]$.
Consequently, we infer $I_{1,\ell}(\rho)=(1-\rho)^\frac32 O(1)$.
We have $\psi_{0,\ell}=a_\ell \psi_{1,\ell}+\tilde a_\ell \tilde \psi_{1,\ell}$ where
$\tilde \psi_{1,\ell}(\rho):=\rho^\ell \tilde \phi_{1,\ell}(\rho^2)$, see Eq.~\eqref{eq:tildephi},
and $a_\ell, \tilde a_\ell \in \C$ are suitable constants.
This yields
\[ \psi_{0,\ell}''(\rho)=(1-\rho)^{-\frac52}O(1)+O(1) \]
and thus, $(1-\rho^2)\psi_{0,\ell}''(\rho)I_{1,\ell}(\rho)=O(1)+(1-\rho)^\frac52 O(1)$.
Consequently, $\rho \mapsto (1-\rho^2)\psi_{0,\ell}''(\rho)I_{1,\ell}(\rho)$ belongs
to $C^1(0,1]$ and by de l'H\^opital's rule we infer $u_{\ell,m}\in C^2(0,1]$
as claimed.

Next, we turn to the endpoint $\rho=0$.
The integrand of $I_{1,\ell}$ is bounded by $C_\ell \rho^{-\ell+1}$ and thus, we obtain
\begin{align*}
|I_{1,\ell}(\rho)|&\lesssim 1,\quad \ell\in \{0,1\} \\
|I_{1,2}(\rho)|&\lesssim |\log \rho| \\
|I_{1,\ell}(\rho)|&\lesssim \rho^{-\ell+2},\quad \ell\in \N,\:3\leq \ell \leq N
\end{align*}
for all $\rho \in (0,1]$.
The integrand of $I_{0,\ell}$ is bounded by $C_\ell \rho^{\ell+2}$ and this implies
$|I_{0,\ell}(\rho)|\lesssim \rho^{\ell+3}$ for all $\rho \in [0,1]$ and $\ell \in \N_0$,
$\ell\leq N$.
Thus, we obtain the estimates
\begin{align}
\label{eq:estu}
|u_{0,m}^{(k)}(\rho)|&\lesssim 1 \nonumber \\
|u_{1,m}^{(k)}(\rho)|&\lesssim \rho^{\max\{1-k,0\}} \nonumber \\
|u_{2,m}^{(k)}(\rho)|&\lesssim \rho^{2-k}|\log \rho|+\rho^{2-k} \nonumber \\
|u_{\ell,m}^{(k)}(\rho)|&\lesssim \rho^{2-k},\quad \ell \in \N,\:3\leq \ell\leq N
\end{align}
for all $\rho \in (0,1]$ and $k\in \{0,1,2\}$.

Now we define the function $u: \overline{B}\backslash \{0\} \to \C$ by  
\begin{equation}
\label{eq:solu}
 u(\xi):=\sum_{\ell,m}^N u_{\ell,m}(|\xi|)Y_{\ell,m}(\tfrac{\xi}{|\xi|}).
 \end{equation}
From the bounds \eqref{eq:estu} we obtain \footnote{
Note that $Y_{0,0}(\omega)=\frac{1}{\sqrt{4\pi}}$.}
$|\partial_j \partial_k u(\xi)|\lesssim |\xi|^{-1}$ 
which implies
 $u\in H^2(B)\cap C^2(\overline B\backslash \{0\})$ and by construction,
$u$ satisfies
 \[ -(\delta^{jk}-\xi^j \xi^k)\partial_j \partial_k u(\xi)+5\xi^j\partial_j u(\xi)
+\tfrac{15}{4}u(\xi)=g_N(\xi) \]
where $\|f-g_N\|_{L^2(B)}\leq \|f-\tilde g\|_{L^2(B)}+\|\tilde g-g_N\|_{L^2(B)}<\varepsilon$.
\end{proof}

\begin{proof}[Proof of Proposition \ref{prop:free}]
First note that $\tilde{\mb L}_0$ is densely defined.
Furthermore, we claim that
\begin{equation}
\label{eq:diss} \Re (\tilde{\mb L}_0 \mb u|\mb u)\leq -\tfrac12 \|\mb u\|^2 
\end{equation}
for all $\mb u\in \mc D(\tilde{\mb L}_0)$.
We write $[\tilde{\mb L}_0 \mb u]_A$ for the $A$-th component of $\tilde{\mb L}_0 \mb u$ where $A\in\{1,2\}$.
Then we have
\[ \partial_k [\tilde{\mb L}_0 \mb u]_1(\xi)=-\xi^j \partial_j \partial_k u_1(\xi)
-2\partial_k u_1(\xi)+\partial_k u_2(\xi). \]
By noting that
\begin{align*}
\Re [\partial_j \partial_k u_1\overline{\partial^k u_1}]&=
\tfrac12 \partial_j [\partial_k u_1\overline{\partial^k u_1}] \\
\xi^j \partial_j f(\xi)&=\partial_{\xi^j}[\xi^j f(\xi)]-3f(\xi)
\end{align*}
we infer 
\[ \Re [\xi^j \partial_j \partial_k u_1(\xi)\overline{\partial^k u_1(\xi)}]
=\tfrac12 \partial_{\xi^j}[\xi^j\partial_k u_1(\xi)\overline{\partial^k u_1(\xi)}]
-\tfrac32 \partial_k u_1(\xi)\overline{\partial^k u_1(\xi)} \]
and the divergence theorem implies
\begin{align*}
\Re \int_B \partial_k [\tilde{\mb L}_0\mb u]_1(\xi)\overline{\partial^k u_1(\xi)}d\xi
&=-\tfrac12 \int_{\partial B}\partial_k u_1(\omega)\overline{\partial^k u_1(\omega)}d\sigma(\omega) \\
&\quad-\tfrac12 \int_B \partial_k u_1(\xi)\overline{\partial^k u_1(\xi)}d\xi \\
&\quad +\Re \int_B \partial_k u_2(\xi)\overline{\partial^k u_1(\xi)}d\xi.
\end{align*}
Furthermore, we have
\begin{align*}
\int_B \partial_j \partial^j u_1(\xi)\overline{u_2(\xi)}d\xi&=\int_{\partial B}
\omega^j \partial_j u_1(\omega)\overline{u_2(\omega)}d\sigma(\omega) \\
&\quad -\int_B \partial_j u_1(\xi)\overline{\partial^j u_2(\xi)}d\xi
\end{align*}
and
\begin{align*}
\Re \int_B \xi^j \partial_j u_2(\xi)\overline{u_2(\xi)}d\xi&=
\tfrac12 \int_{\partial B}|u_2(\omega)|^2 d\sigma(\omega) 
-\tfrac32\int_B|u_2(\xi)|^2 d\xi 
\end{align*}
which yields
\begin{align*}
\Re \int_B [\tilde{\mb L}_0\mb u]_2(\xi)\overline{u_2(\xi)}d\xi&=
\Re \int_{\partial B}\omega^j \partial_j u_1(\omega)\overline{u_2(\omega)}d\sigma(\omega)
-\tfrac12 \|u_2\|_{L^2(\partial B)}^2 \\
&\quad -\Re \int_B \partial_j u_1(\xi)\overline{\partial^j u_2(\xi)}d\xi-\tfrac12 \|u_2\|_{L^2(B)}^2.
\end{align*}
In summary, we infer
\[
\Re (\tilde{\mb L}_0 \mb u|\mb u)=-\tfrac12 \|u_1\|_{\dot H^1(B)}^2-\tfrac12 \|u_2\|_{L^2(B)}^2 
+\int_{\partial B}A(\omega)d\sigma(\omega) \]
with 
\begin{align*}
A(\omega)&=-\tfrac12 |u_1(\omega)|^2- \tfrac12 |u_1(\omega)|^2-\tfrac12 |\nabla u_1(\omega)|^2
-\tfrac12 |u_2(\omega)|^2 \\
&\quad -\Re [\omega^j\partial_j u_1(\omega)
\overline{u_1(\omega)}]
+\Re [\omega^j \partial_j u_1(\omega)\overline{u_2(\omega)}]
+\Re [u_2(\omega)\overline{u_1(\omega)}] \\
&\leq -\tfrac12 |u_1(\omega)|^2
\end{align*}
where we have used the inequality
\[ \Re(\overline a b)+\Re(\overline a c)-\Re(\overline b c)\leq \tfrac12 (|a|^2+|b|^2+|c|^2),
\quad a,b,c \in \C \]
which follows from $0\leq |a-b-c|^2$.
This proves \eqref{eq:diss}.

The estimate \eqref{eq:diss} implies
\begin{align*} \big \|[\lambda-(\tilde{\mb L}_0+\tfrac12)]\mb u \big \|^2&=\lambda^2 \|\mb u\|^2
-2\lambda \Re \big ((\tilde{\mb L}_0+\tfrac12)\mb u  \big|\mb u \big)+\|(\tilde{\mb L}_0+\tfrac12)\mb u\|^2  \\
&\geq \lambda^2 \|\mb u\|^2
\end{align*}
for all $\lambda>0$ and $\mb u\in \mc D(\tilde{\mb L}_0)$.
Thus, in view of the Lumer-Phillips theorem (\cite{EngNag00}, p.~83, Theorem 3.15)
it suffices to prove density of the range
of $\lambda-\tilde{\mb L}_0$ for some $\lambda>-\frac12$.
Let $\mb f \in \mc H$ and $\varepsilon>0$ be arbitrary.
We consider the equation $(\lambda-\tilde{\mb L}_0)\mb u=\mb f$.
From the first component we infer $u_2=\xi^j \partial_j u_1+(\lambda+1)u_1-f_1$ and inserting
this in the second component we arrive at the degenerate elliptic problem
\begin{equation}
\label{eq:degelllambda}
 -(\delta^{jk}-\xi^j \xi^k)\partial_j \partial_k u(\xi)+2(\lambda+2)\xi^j\partial_j u(\xi)
+(\lambda+1)(\lambda+2)u(\xi)=f(\xi) 
\end{equation}
for $u=u_1$ and $f(\xi):=\xi^j \partial_j f_1(\xi)+(\lambda+2)f_1(\xi)+f_2(\xi)$.
Note that by assumption we have $f \in L^2(B)$.
Setting $\lambda=\frac12$ we infer from Lemma \ref{lem:degell} the existence of 
functions $u \in H^2(B)\cap C^2(\overline B\backslash \{0\})$ and $g\in L^2(B)$ such that
\[  -(\delta^{jk}-\xi^j \xi^k)\partial_j \partial_k u(\xi)+5\xi^j\partial_j u(\xi)
+\tfrac{15}{4}u(\xi)=g(\xi) \]
and $\|f-g\|_{L^2(B)}<\varepsilon$.
We set $u_1:=u$, $u_2(\xi):=\xi^j\partial_j u(\xi)+\frac32 u-f_1$,
$g_1:=f_1$, and $g_2(\xi):=g(\xi)-\xi^j \partial_j f_1(\xi)-\frac52 f_1(\xi)$.
Then we have $\mb u \in \mc D(\tilde{\mb L}_0)$, $\mb g \in \mc H$,
\[ \|\mb f-\mb g\|=\|f_2-g_2\|_{L^2(B)}=\|f-g\|_{L^2(B)}<\varepsilon \]
and by construction, $(\frac12-\tilde{\mb L}_0)\mb u=\mb g$.
Since $\mb f \in \mc H$ and $\varepsilon>0$ were arbitrary, this 
shows that $\rg (\frac12-\tilde{\mb L}_0)$ is dense in $\mc H$ which finishes the proof.
\end{proof}

\subsection{Well-posedness for the linearized problem}

Next, we include the potential term and consider the system
\begin{align}
\label{eq:lincss}
\partial_0 \phi_1&=-\xi^j \partial_j \phi_1-\phi_1+\phi_2 \nonumber \\
\partial_0 \phi_2&=\partial_j \partial^j \phi_1-\xi^j \partial_j \phi_2-2\phi_2
+6\phi_1.
\end{align}
We define an operator $\mb L'$, acting on $\mc H$, by
\[ \mb L' \mb u(\xi):=\left (\begin{array}{c}0 \\ 6u_1 \end{array} \right ). \]
Then we may rewrite Eq.~\eqref{eq:lincss} as an ODE
\[ \tfrac{d}{d\tau}\Phi(\tau)=(\mb L_0+\mb L')\Phi(\tau) \]
for a function $\Phi: [0,\infty)\to \mc H$.

\begin{lemma}
\label{lem:lin}
The operator $\mb L:=\mb L_0+\mb L': \mc D(\mb L_0)\subset \mc H\to \mc H$ 
generates a strongly-continuous
one-parameter semigroup $\mb S: [0,\infty)\to \mc B(\mc H)$ satisfying $\|\mb S(\tau)\|
\leq e^{(-\frac12+\|\mb L'\|)\tau}$.
Furthermore, for the spectrum of the generator we have
$\sigma(\mb L)\backslash \sigma(\mb L_0)=\sigma_p(\mb L)$.
\end{lemma}

\begin{proof}
The first assertion is an immediate consequence of the Bounded Perturbation Theorem of semigroup theory, see
\cite{EngNag00}, p.~158, Theorem 1.3.
In order to prove the claim about the spectrum, we note that the operator $\mb L':
\mc H\to \mc H$ is compact by the compactness of the embedding $H^1(B)\hookrightarrow L^2(B)$
(Rellich-Kondrachov) and Lemma \ref{lem:equivH1}.
Assume that $\lambda \in \sigma(\mb L)$ and $\lambda \notin \sigma(\mb L_0)$.
Then we may write
$\lambda-\mb L=[1-\mb L'\mb R_{\mb L_0}(\lambda)](\lambda-\mb L_0)$.
Observe that the operator $\mb L'\mb R_{\mb L_0}(\lambda)$ is compact.
Furthermore, $1 \in \sigma(\mb L'\mb R_{\mb L_0}(\lambda))$ since otherwise we would have $\lambda \in
\rho(\mb L)$, a contradiction to our assumption.
By the spectral theorem for compact operators we infer $1 \in \sigma_p(\mb L'\mb R_{\mb L_0}(\lambda))$
which shows that there exists a nontrivial $\mb f \in \mc H$ such that $[1-\mb L' \mb R_{\mb L_0}(\lambda)]\mb f
=\mb 0$. Thus, by setting $\mb u:=\mb R_{\mb L_0}(\lambda)\mb f$, we infer 
$\mb u \in \mc D(\mb L_0)$,
$\mb u\not= \mb 0$, and $(\lambda-\mb L)\mb u=\mb 0$ which implies $\lambda \in \sigma_p(\mb L)$.
\end{proof}

\subsection{Spectral analysis of the generator}
In order to improve the rough growth bound for $\mb S$ given in Lemma \ref{lem:lin}, we 
need more information on the spectrum of $\mb L$.
Thanks to Lemma \ref{lem:lin} we are only concerned with point spectrum.
To begin with, we need the following result concerning $\mc D(\mb L_0)$.

\begin{lemma}
\label{lem:DL0}
Let $\delta \in (0,1)$ and $\mb u\in \mc D(\mb L_0)$. Then 
\[ \mb u|_{B_{1-\delta}} \in 
H^2(B_{1-\delta})\times H^1(B_{1-\delta}) \] where $B_{1-\delta}:=\{\xi\in \R^3: |\xi|<1-\delta\}$.
\end{lemma}

\begin{proof}
Let $\mb u\in \mc D(\mb L_0)$. By definition of the closure there exists a sequence
$(\mb u_n)\subset \mc D(\tilde{\mb L}_0)$ such that $\mb u_n \to \mb u$ and $\tilde{\mb L}_0\mb u_n
\to  \mb L_0 \mb u$ in $\mc H$ as $n\to\infty$.
We set $\mb f_n:=\tilde{\mb L}_0 \mb u_n$ and note that $\mb f_n \in H^1(B)\cap C^1(B)\times
L^2(B)\cap C(B)$ for all $n\in \N$ by the definition of $\mc D(\tilde{\mb L}_0)$.
We obtain $u_{2n}(\xi)=\xi^j \partial_j u_{1n}(\xi)+u_{1n}(\xi)+f_{1n}(\xi)$ and
\begin{equation}
\label{eq:degell0}
 -(\delta^{jk}-\xi^j \xi^k)\partial_j \partial_k u_{1n}(\xi)+
 4\xi^j \partial_j u_{1n}(\xi)+2u_{1n}(\xi)=f_n(\xi) 
 \end{equation}
where $f_n(\xi):=-\xi^j \partial_j f_{1n}(\xi)-2f_{1n}(\xi)-f_{2n}(\xi)$, 
cf.~Eq.~\eqref{eq:degelllambda}.
By assumption we have $f_n \to f$ in $L^2(B)$ for some $f\in L^2(B)$.
Since 
\[ (\delta^{jk}-\xi^j \xi^k)\eta_j \eta_k\geq |\eta|^2-|\xi|^2|\eta|^2\geq 
\tfrac{\delta}{2} |\eta|^2 \]
for all $\xi \in B_{1-\delta/2}$ and all $\eta \in \R^3$, we see that the differential operator in Eq.~\eqref{eq:degell0}
is uniformly elliptic on $B_{1-\delta/2}$.
Thus, by standard elliptic regularity theory (see \cite{Eva98}, p.~309, Theorem 1) 
we obtain the estimate 
\begin{equation}
\label{eq:estell} \|u_{1n}\|_{H^2(B_{1-\delta})}\leq C_\delta \left (\|u_{1n}\|_{L^2(B_{1-\delta/2})}
+\|f_n\|_{L^2(B_{1-\delta/2})} \right ) 
\end{equation}
and since $\mb u_n \to \mb u$ in $\mc H$ implies $u_{1n} \to u_1$ in $L^2(B)$, we infer 
$u_1|_{B_{1-\delta}}\in H^2(B_{1-\delta})$.
Finally, from $u_{2n}(\xi)=\xi^j \partial_j u_{1n}(\xi)+u_{1n}(\xi)+f_{1n}(\xi)$ we conclude
$u_2|_{B_{1-\delta}}\in H^1(B_{1-\delta})$.
\end{proof}

The next result allows us to obtain information on the spectrum of $\mb L$ by
studying an ODE.
For the following we define the space $H_\mathrm{rad}^1(a,b)$ by
\[ \|f\|_{H_\mathrm{rad}^1(a,b)}^2:=\int_a^b |f'(\rho)|^2 \rho^2 d\rho
+\int_a^b |f(\rho)|^2 \rho^2 d\rho. \]

\begin{lemma}
\label{lem:specODE}
Let $\lambda \in \sigma_p(\mb L)$. Then there exists an $\ell \in \N_0$ and a
nonzero function $u\in C^\infty(0,1)\cap H_\mathrm{rad}^1(0,1)$ such that
\begin{align}
\label{eq:specl} 
-(1-\rho^2)u''(\rho)&-\tfrac{2}{\rho}u'(\rho) 
+2(\lambda+2)\rho u'(\rho) \nonumber \\
&+[(\lambda+1)(\lambda+2)-6]u(\rho) 
+\tfrac{\ell(\ell+1)}{\rho^2}u(\rho)=0 
\end{align}
for all $\rho \in (0,1)$.
\end{lemma}

\begin{proof}
Let $\mb u \in \mc D(\mb L)=\mc D(\mb L_0)$ be an eigenvector associated to the eigenvalue
$\lambda \in \sigma_p(\mb L)$.
The spectral equation $(\lambda-\mb L)\mb u=\mb 0$ 
implies $u_2(\xi)=\xi^j \partial_j u_1(\xi)
+(\lambda+1)u_1(\xi)$ and
\begin{align} 
\label{eq:u1spec}
-(\delta^{jk}-\xi^j\xi^k)\partial_j \partial_k u_1(\xi)&+2(\lambda+2)\xi^j 
\partial_j u_1(\xi) \nonumber \\
&+[(\lambda+1)(\lambda+2)-6]u_1(\xi)=0, 
\end{align}
cf.~Eq.~\eqref{eq:degelllambda}, but this time the derivatives have to be interpreted in
the weak sense since a priori we merely have $u_1 \in H^2(B_{1-\delta})\cap H^1(B)$
and $u_2 \in H^1(B_{1-\delta})\cap L^2(B)$ by Lemma \ref{lem:DL0}.
However, by invoking elliptic regularity theory (\cite{Eva98}, p.~316, Theorem 3)
we see that in fact $u_1\in C^\infty(B)\cap H^1(B)$.
As always, we write $\rho=|\xi|$ and $\omega=\frac{\xi}{|\xi|}$.
We expand $u_1$ in spherical harmonics, i.e.,
\begin{equation}
\label{eq:expu1}
 u_1(\rho \omega)=\sum_{\ell,m}^\infty u_{\ell,m}(\rho) Y_{\ell,m}(\omega) 
 \end{equation}
with
$u_{\ell,m}(\rho)=(u_1(\rho\:\cdot)|Y_{\ell,m})_{L^2(S^2)}$
and for each fixed $\rho \in (0,1)$, the sum converges in $L^2(S^2)$.
By dominated convergence and $u_1\in C^\infty(B)$ it follows that $u_{\ell,m}\in C^\infty(0,1)$.
Similarly, we may expand $\partial_\rho u_1(\rho\omega)$ in 
spherical harmonics. The corresponding expansion coefficients are given by 
\[ (\partial_\rho u_1(\rho\,\cdot)|Y_{\ell,m})_{L^2(S^2)}=
\partial_\rho (u_1(\rho\,\cdot)|Y_{\ell,m})_{L^2(S^2)}
=\partial_\rho u_{\ell,m}(\rho) \]
where we used dominated convergence and the smoothness of $u_1$ to pull out the derivative $\partial_\rho$ of the inner
product.
In other words, we may interchange the operator $\partial_\rho$ with the sum in Eq.~\eqref{eq:expu1}.
Analogously, we may expand $\Delta_{S^2}u_1(\rho\, \cdot)$ and
the corresponding expansion coefficients are
\[ (\Delta_{S^2} u_1(\rho\,\cdot)|Y_{\ell,m})_{L^2(S^2)}
=(u_1(\rho\,\cdot)|\Delta_{S^2}Y_{\ell,m})_{L^2(S^2)}
=-\ell(\ell+1)u_{\ell,m}(\rho). \]
Thus, the operator $\Delta_{S^2}$ commutes with the sum in \eqref{eq:expu1}.
All differential operators that appear in Eq.~\eqref{eq:u1spec} are composed of $\partial_\rho$
and $\Delta_{S^2}$ and it is
therefore a consequence of Eq.~\eqref{eq:u1spec} that each $u_{\ell,m}$ satisfies
Eq.~\eqref{eq:specl}
for all $\rho\in (0,1)$.
Since at least one $u_{\ell,m}$ is nonzero, we obtain the desired function $u \in C^\infty(0,1)$.
To complete the proof, it remains to show that $u_{\ell,m} \in H^1_\mathrm{rad}(0,1)$.
We have
\[ |u_{\ell,m}(\rho)|=\left |(u_1(\rho\,\cdot)|Y_{\ell,m})_{L^2(S^2)}\right |
\leq \|u_1(\rho\,\cdot)\|_{L^2(S^2)} \]
and thus,
\begin{align*} 
\int_0^1 |u_{\ell,m}(\rho)|^2\rho^2 d\rho
\leq \int_0^1 \|u_1(\rho\,\cdot)\|_{L^2(S^2)}^2 \rho^2 d\rho=\|u_1\|_{L^2(B)}^2.
\end{align*}
Similarly, by dominated convergence,
\[ |\partial_\rho u_{\ell,m}(\rho)|=\left |(\partial_\rho u_1(\rho\,\cdot)|Y_{\ell,m})_{L^2(S^2)}
\right |
\lesssim \|\nabla u_1(\rho\,\cdot)\|_{L^2(S^2)} \]
and thus,
\[ \int_0^1 |u_{\ell,m}'(\rho)|^2 \rho^2 d\rho \lesssim \|\nabla u_1\|_{L^2(B)}^2. \]
Consequently, $u_1 \in H^1(B)$ implies $u_{\ell,m} \in H^1_\mathrm{rad}(0,1)$.
\end{proof}

\begin{proposition}
\label{prop:spec}
For the spectrum of $\mb L$ we have
\[ \sigma(\mb L) \subset \{z\in \C: \Re z\leq -\tfrac12\}\cup \{0,1\}. \]
Furthermore, $\{0,1\} \subset \sigma_p(\mb L)$ and the (geometric) eigenspace of the eigenvalue $1$
is one-dimensional and spanned by
\[ \mb u(\xi;1)=\left (\begin{array}{c}1 \\ 2 \end{array} \right ) \]
whereas the (geometric) eigenspace of the eigenvalue $0$ is three-dimensional and spanned by
\[ \mb u_j(\xi;0)=\left ( \begin{array}{c}\xi^j \\ 2\xi^j \end{array} \right ),\quad 
j\in \{1,2,3\}. \]
\end{proposition}

\begin{proof}
First of all, it is a simple exercise to check that $\mb L \mb u(\xi,1)=\mb u(\xi;1)$
and $\mb L \mb u_j(\xi; 0)=\mb 0$ for $j=1,2,3$.
Since obviously $\mb u(\cdot;1),\mb u_j(\cdot;0) \in \mc D(\tilde{\mb L}_0)$, this implies
$\{0,1\}\subset \sigma_p(\mb L)$.

In order to prove the first assertion,
let $\lambda \in \sigma(\mb L)$ and assume $\Re\lambda>-\frac12$.
By Proposition \ref{prop:free} we have $\lambda \notin \sigma(\mb L_0)$ and thus,
Lemma \ref{lem:lin} implies $\lambda \in \sigma_p(\mb L)$.
From Lemma \ref{lem:specODE} we infer the existence of a nonzero $u \in C^\infty(0,1)
\cap H^1_{\mathrm{rad}}(0,1)$ satisfying Eq.~\eqref{eq:specl} for $\rho\in (0,1)$.
As before, we reduce Eq.~\eqref{eq:specl} to the hypergeometric differential equation
by setting $u(\rho)=\rho^\ell v(\rho^2)$. This yields
\begin{equation}
\label{eq:hypgeolambda}
 z(1-z)v''(z)+[c-(a+b+1)z]v'(z)-ab v(z)=0 
 \end{equation}
with $a=\frac12(-1+\ell+\lambda)$, $b=\frac12(4+\ell+\lambda)$, $c=\frac32+\ell$, and
$z=\rho^2$.
A fundamental system of Eq.~\eqref{eq:hypgeolambda} is given by 
\footnote{Strictly speaking, this is only true for $c-a-b=-\lambda\not= 0$.
In the case $\lambda=0$ there exists a solution $\tilde \phi_{1,\ell}$ which behaves
like $\log(1-z)$ as $z\to 1-$.}
\begin{align*}
\phi_{1,\ell}(z;\lambda)&={}_2F_1(a,b,a+b+1-c;1-z) \\
\tilde \phi_{1,\ell}(z;\lambda)&=(1-z)^{c-a-b}{}_2F_1(c-a,c-b,c-a-b+1; 1-z)
\end{align*}
and thus, there exist constants $c_\ell(\lambda)$ and $\tilde c_\ell(\lambda)$ such that
\[ v(z)=c_\ell(\lambda)\phi_{1,\ell}(z;\lambda)
+\tilde c_\ell(\lambda)\tilde \phi_{1,\ell}(z;\lambda). \]
The function $\phi_{1,\ell}(z;\lambda)$ is analytic around $z=1$ whereas 
$\tilde \phi_{1,\ell}(z;\lambda)\sim (1-z)^{-\lambda}$ as $z \to 1-$ provided $\lambda \not= 0$.
In the case $\lambda=0$ we have $\tilde \phi_{1,\ell}(z;\lambda)\sim \log(1-z)$ as $z\to 1-$.
Since $u\in H^1_\mathrm{rad}(0,1)$ implies $v\in H^1(\frac12,1)$
and we assume $\Re\lambda>-\frac12$, it follows that $\tilde c_\ell(\lambda)=0$.
Another fundamental system of Eq.~\eqref{eq:hypgeolambda} is given by
\begin{align*} 
\phi_{0,\ell}(z;\lambda)&={}_2F_1(a,b,c;z) \\
\tilde \phi_{0,\ell}(z;\lambda)&=z^{1-c}{}_2F_1(a-c+1,b-c+1,2-c;z)
\end{align*}
and since $\tilde \phi_{0,\ell}(z;\lambda)\sim z^{-\ell-\frac12}$ as $z\to 0+$, we
see that the function $\rho\mapsto \rho^\ell \tilde \phi_{\ell,0}(\rho^2)$
does not belong to $H^1_\mathrm{rad}(0,\frac12)$.
As a consequence, we must have $v(z)=d_\ell(\lambda)\phi_{0,\ell}(z;\lambda)$ for
some suitable $d_\ell(\lambda)\in \C$.
In summary, we conclude that the functions $\phi_{0,\ell}(\cdot;\lambda)$ and 
$\phi_{1,\ell}(\cdot;\lambda)$ are linearly dependent and in view of the connection
formula \cite{DLMF}
\begin{align*} 
\phi_{1,\ell}(z;\lambda)&=\frac{\Gamma(1-c)\Gamma(a+b+1-c)}{\Gamma(a+1-c)\Gamma(b+1-c)}
\phi_{0,\ell}(z;\lambda) \\
&\quad+\frac{\Gamma(c-1)\Gamma(a+b+1-c)}{\Gamma(a)\Gamma(b)}
\tilde \phi_{0,\ell}(z;\lambda) 
\end{align*}
this is possible only if $a$ or $b$ is a pole of the $\Gamma$-function.
This yields $-a\in \N_0$ or $-b\in \N_0$ and thus,
$\tfrac12(1-\ell-\lambda) \in \N_0$ or $-\tfrac12(4+\ell+\lambda)\in \N_0$.
The latter condition is not satisfied for any $\ell \in \N_0$ and the former one
is satisfied only if $(\ell,\lambda)=(0,1)$ or $(\ell,\lambda)=(1,0)$
which proves $\sigma(\mb L)\subset \{z\in \Re z\leq -\frac12\}\cup \{0,1\}$.
Furthermore, the above argument and the derivation in the proof of Lemma \ref{lem:specODE}
also show that the geometric eigenspaces of the eigenvalues $0$ and $1$ are at most
three- and one-dimensional, respectively.
\end{proof}

\begin{remark}
According to the discussion at the beginning of Section \ref{sec:lin}, the two unstable
eigenvalues $1$ and $0$ emerge from the time translation and Lorentz invariance of the
wave equation. 
\end{remark}

\subsection{Spectral projections}

In order to force convergence to the attractor, we need to ``remove'' the eigenvalues
$0$ and $1$ from the spectrum of $\mb L$. 
This is achieved by the spectral projection
\begin{equation}
\label{eq:P} 
\mb P:=\frac{1}{2\pi i}\int_\gamma (z-\mb L)^{-1}dz 
\end{equation}
where the contour $\gamma$ is given by the curve $\gamma(s)=\frac12+\frac34e^{2\pi i s}$, $s \in [0,1]$.
By Proposition \ref{prop:spec} it follows that $\gamma(s)\in \rho(\mb L)$ for all $s\in [0,1]$
and thus, the integral in Eq.~\eqref{eq:P} is well-defined as a Riemann integral over a
continuous function (with values in a Banach space, though).
Furthermore, the contour $\gamma$ encloses the two unstable eigenvalues $0$ and $1$.
The operator $\mb L$ decomposes into two parts 
\begin{align*} 
\mb L_u&: \rg \mb P\cap \mc D(\mb L)\to \mc \rg \mb P, & \mb L_u\mb u&=\mb L\mb u \\
\mb L_s&: \ker \mb P \cap \mc D(\mb L)\to \ker \mb P, &\mb L_s\mb u&=\mb L\mb u
\end{align*}
and for the spectra we have $\sigma(\mb L_u)=\{0,1\}$ as well as $\sigma(\mb L_s)=
\sigma(\mb L)\backslash \{0,1\}$.
We also emphasize the crucial fact that $\mb P$ commutes with the semigroup
$\mb S(\tau)$ and thus, the subspaces $\rg \mb P$ and $\ker \mb P$ of $\mc H$ are invariant
under the linearized flow.
We refer to \cite{Kat95} and \cite{EngNag00} for these standard facts.
However, it is important to keep in mind that $\mb P$ is not an orthogonal projection
since $\mb L$ is not self-adjoint.
Consequently, the following statement on the dimension of $\rg \mb P$ is not trivial.

\begin{lemma}
\label{lem:dimP}
The algebraic multiplicities of the eigenvalues $0, 1 \in \sigma_p(\mb L)$ equal
their geometric multiplicities. In particular, we have $\dim \rg \mb P=4$.
\end{lemma}

\begin{proof}
We define the two spectral projections $\mb P_0$ and $\mb P_1$ by
\[ \mb P_n=\frac{1}{2\pi i}\int_{\gamma_n}(z-\mb L)^{-1}dz,\quad n \in \{0,1\} \]
where $\gamma_0(s)=\frac12 e^{2\pi i s}$ and $\gamma_1(s)=1+\frac12 e^{2\pi i s}$
for $s\in [0,1]$.
Note that $\mb P=\mb P_0+\mb P_1$ and $\mb P_0\mb P_1=\mb P_1 \mb P_0=\mb 0$, see \cite{Kat95}.
By definition, the algebraic multiplicity of the eigenvalue $n \in \sigma_p(\mb L)$ equals
$\dim \rg \mb P_n$.
First, we exclude the possibility $\dim \rg \mb P_n=\infty$.
Suppose this is true. Then $n$ belongs to the essential spectrum of $\mb L$, i.e.,
$n-\mb L$ fails to be semi-Fredholm (\cite{Kat95}, p.~239, Theorem 5.28).
Since the essential spectrum is invariant under compact perturbations, see \cite{Kat95},
p.~244, Theorem 5.35,
we infer $n \in \sigma(\mb L_0)$ which contradicts the spectral statement in Proposition \ref{prop:free}.
Consequently, $\dim \rg \mb P_n<\infty$.
We conclude that the operators $\mb L_{(n)}:=\mb L|_{\rg \mb P_n \cap \mc D(\mb L)}$ are in fact 
finite-dimensional and $\sigma(\mb L_{(n)})=\{n\}$.
This implies that $n-\mb L_{(n)}$ is nilpotent and thus, there exist
$m_n \in \N$ such that $(n-\mb L_{(n)})^{m_n}=\mb 0$.
We assume $m_n$ to be minimal with this property.
If $m_n=1$ we are done.
Thus, assume $m_n\geq 2$.
We first consider $\mb L_{(0)}$.
Since $\ker \mb L$ is spanned by $\{\mb u_j(\cdot; 0): j=1,2,3\}$ by Proposition 
\ref{prop:spec}, it follows that there exists a $\mb u \in \rg \mb P_0 \cap \mc D(\mb L)$ 
and constants $c_1,c_2,c_3 \in \C$, not all of them zero, such
that 
\[ \mb L_{(0)}\mb u(\xi)=\mb L\mb u(\xi)=\sum_{j=1}^3 c_j \mb u_j(\xi; 0)=
\left (\begin{array}{c}c_j\xi^j \\ 2c_j\xi^j \end{array} \right ). \]
This implies $u_2(\xi)=\xi^j \partial_j u_1(\xi)+u_1(\xi)+c_j \xi^j$ and thus, 
\begin{align*} 
-(\delta^{jk}-\xi^j \xi^k)\partial_j \partial_k u_1(\xi) +4\xi^j \partial_j u_1(\xi)
-4u_1(\xi)&=-5c_j\xi^j \\
&=|\xi|\sum_{m=-1}^1 \tilde c_m Y_{1,m}(\tfrac{\xi}{|\xi|}).
\end{align*}
As before in the proof of Lemma \ref{lem:specODE}, 
we expand $u_1$ as
\[ u_1(\xi)=\sum_{\ell,m}^\infty u_{\ell,m}(|\xi|)Y_{\ell,m}(\tfrac{\xi}{|\xi|}) \]
and find
\begin{equation}
\label{eq:projode}
 -(1-\rho^2)u_{1,m}''(\rho)-\tfrac{2}{\rho}u_{1,m}'(\rho)+4\rho
u_{1,m}'(\rho)-4u_{1,m}(\rho)+\tfrac{2}{\rho^2}u_{1,m}(\rho)=\tilde c_m \rho. 
\end{equation}
For at least one $m \in \{-1,0,1\}$ we have $\tilde c_m\not= 0$ and by normalizing
$u_{1,m}$ accordingly, we may assume $\tilde c_m=1$.
Of course, Eq.~\eqref{eq:projode} with $\tilde c_m=0$ is nothing but the spectral equation
\eqref{eq:specl} with $\ell=1$ and $\lambda=0$.
An explicit solution is therefore given by $\psi(\rho)=\rho$ which may of course
also be easily checked directly.
Another solution is $\tilde \psi(\rho):=\tilde \psi_{0,1}(\rho;0)=\rho \tilde \phi_{0,1}(\rho^2; 0)$ where
$\tilde \phi_{1,0}(\cdot; 0)$ is the hypergeometric function from the proof of
Lemma \ref{prop:spec}.
We have the asymptotic behavior $\tilde \psi(\rho)\sim \rho^{-2}$ as $\rho \to 0+$
and $|\tilde \psi(\rho)|\simeq |\log(1-\rho)|$ as $\rho \to 1-$.
By the variation of constants formula we infer that $u_{1,m}$ must be of the form
\begin{align} 
\label{eq:u1mvdkf}
u_{1,m}(\rho)=&c \psi(\rho)+\tilde c \tilde \psi(\rho)
+\psi(\rho)\int_{\rho_0}^\rho \frac{\tilde \psi(s)}{W(s)}\frac{s}{1-s^2} ds \nonumber \\
&-\tilde \psi(\rho)\int_{\rho_1}^\rho \frac{\psi(s)}{W(s)}\frac{s}{1-s^2} ds 
\end{align}
for suitable constants $c,\tilde c\in \C$, $\rho_0,\rho_1 \in [0,1]$ and
\[ W(\rho)=W(\psi,\tilde \psi)(\rho)=\frac{d}{\rho^2(1-\rho^2)} \]
where $d \in \R\backslash \{0\}$.
Recall that $u_1 \in H^1(B)$ implies $u_{1,m}\in H^1_{\mathrm{rad}}(0,1)$ and by considering
the behavior of \eqref{eq:u1mvdkf} as $\rho \to 0+$, we see that necessarily 
\[ \tilde c=\int_{\rho_1}^0 \frac{\psi(s)}{W(s)}\frac{s}{1-s^2}ds \]
which leaves us with
\[ u_{1,m}(\rho)=c\psi(\rho)+\psi(\rho)\int_{\rho_0}^\rho \frac{\tilde \psi(s)}{W(s)}\frac{s}{1-s^2}ds
-\tilde \psi(\rho)\int_0^\rho \frac{\psi(s)}{W(s)}\frac{s}{1-s^2}ds. \]
Next, we consider the behavior as $\rho \to 1-$.
Since 
\[ \left |\int_{\rho_0}^\rho \frac{\tilde \psi(s)}{W(s)}\frac{s}{1-s^2}ds \right |\lesssim 1 \]
for all $\rho \in (0,1)$ and $\tilde \psi \notin H^1_\mathrm{rad}(\frac12,1)$, we must
have
\[ \lim_{\rho\to 1-}\int_0^\rho \frac{\psi(s)}{W(s)}\frac{s}{1-s^2}ds=0. \]
This, however, is impossible since $\frac{\psi(s)}{W(s)}\frac{s}{1-s^2}=\frac{1}{d}s^4$. 
Thus, we arrive at a contradiction and our initial assumption $m_0\geq 2$ must be wrong.
Consequently, from Proposition \ref{prop:spec} we infer
$\dim \rg \mb P_0=\dim \ker \mb L=3$ as claimed.
By exactly the same type of argument one proves that $\dim \rg \mb P_1=1$.
\end{proof}

\subsection{Resolvent estimates}

Our next goal is to obtain existence of the resolvent $\mb R_\mb L(\lambda) \in \mc B(\mc H)$ for
$\lambda \in H_{-\frac12+\epsilon}:=\{z\in \C: \Re z\geq -\frac12+\epsilon\}$ and $|\lambda|$
large.

\begin{lemma}
\label{lem:res}
Fix $\epsilon>0$. Then there exists a constant $C>0$ such that $\mb R_{\mb L}(\lambda)$
exists as a bounded operator on $\mc H$ for all $\lambda \in H_{-\frac12+\epsilon}$
with $|\lambda|>C$.
\end{lemma}

\begin{proof}
From Proposition \ref{prop:free} we know that $\mb R_{\mb L_0}(\lambda) \in \mc B(\mc H)$
for all $\lambda \in H_{-\frac12+\epsilon}$ with the bound
\[ \|\mb R_{\mb L_0}(\lambda)\|\leq \frac{1}{\Re \lambda+\frac12}, \]
see \cite{EngNag00}, p.~55, Theorem 1.10.
Furthermore, recall the identity $\mb R_{\mb L}(\lambda)=\mb R_{\mb L_0}(\lambda)[1-\mb L'
\mb R_{\mb L_0}(\lambda)]^{-1}$.
By definition of $\mb L'$ we have
\[ \mb L'\mb R_{\mb L_0}(\lambda)\mb f=\left (\begin{array}{c}
0 \\ 6[\mb R_{\mb L_0}(\lambda)\mb f]_1 \end{array} \right ) \]
where we use the notation $[\mb g]_k$ for the $k$-th component of the vector $\mb g$.
Set $\mb u=\mb R_{\mb L_0}(\lambda)\mb f$ for a given $\mb f\in \mc H$.
Then we have $\mb u\in \mc D(\mb L_0)$ and $(\lambda-\mb L_0)\mb u=\mb f$ which implies
$u_2(\xi)=\xi^j \partial_j u_1(\xi)+(\lambda+1)u_1(\xi)-f_1(\xi)$, or, equivalently,
\[ [\mb R_{\mb L_0}(\lambda)\mb f]_1(\xi)=\frac{1}{\lambda+1}\Big [
-\xi^j \partial_j [\mb R_{\mb L_0}(\lambda)\mb f]_1(\xi) 
+[\mb R_{\mb L_0}(\lambda)\mb f]_2(\xi)+f_1(\xi)
\Big ]. \]
Consequently, we infer
\begin{align*} 
\|[\mb R_{\mb L_0}(\lambda)\mb f]_1\|_{L^2(B)}&\lesssim \frac{1}{|\lambda+1|}
\Big [
\|[\mb R_{\mb L_0}(\lambda)\mb f]_1\|_{H^1(B)}+\|[\mb R_{\mb L_0}(\lambda)\mb f]_2\|_{L^2(B)}\\
&\quad +\|f_1\|_{L^2(B)} \Big ] \\
&\lesssim \frac{\|\mb f\|}{|\lambda+1|}
\end{align*}
which yields $\|\mb L' \mb R_{\mb L_0}(\lambda)\|\lesssim \frac{1}{|\lambda+1|}$
for all $\lambda \in H_{-\frac12+\epsilon}$.
We conclude that the Neumann series
\[ [1-\mb L'\mb R_{\mb L_0}(\lambda)]^{-1}=\sum_{k=0}^\infty [\mb L' \mb R_{\mb L_0}(\lambda)]^k \]
converges in norm provided $|\lambda|$ is sufficiently large.
This yields the desired result.
\end{proof}

\subsection{Estimates for the linearized evolution}

Finally, we obtain improved growth estimates for the semigroup $\mb S$ from Lemma \ref{lem:lin}
which governs the linearized evolution.

\begin{proposition}
\label{prop:lin}
Fix $\epsilon>0$.
Then the semigroup $\mb S$ from Lemma \ref{lem:lin} satisfies the estimates
\begin{align*} 
\|\mb S(\tau)(1-\mb P)\mb f\|&\leq Ce^{(-\frac12+\epsilon)\tau}\|(1-\mb P)\mb f\| \\
\|\mb S(\tau)\mb P\mb f\|&\leq Ce^\tau \|\mb P\mb f\|
\end{align*}
for all $\tau\geq 0$ and $\mb f\in \mc H$.
\end{proposition}

\begin{proof}
The operator $\mb L_s$ is the generator of the subspace semigroup $\mb S_s$ defined by
$\mb S_s(\tau):=\mb S(\tau)|_{\ker \mb P}$.
We have $\sigma(\mb L_s)\subset \{z\in \C: \Re z\leq -\frac12\}$
and the resolvent $\mb R_{\mb L_s}(\lambda)$ is the restriction of 
$\mb R_{\mb L}(\lambda)$ to $\ker \mb P$.
Consequently, by Lemma \ref{lem:res} we infer
$\|\mb R_{\mb L_s}(\lambda)\|\lesssim 1$ for all $\lambda \in H_{-\frac12+\epsilon}$ and
thus, the Gearhart-Pr\"uss-Greiner Theorem (see \cite{EngNag00}, p.~302, Theorem 1.11)
yields the semigroup decay
$\|\mb S_s(\tau)\|\lesssim e^{(-\frac12+\epsilon)\tau}$.
The estimate for $\mb S(\tau)\mb P$ follows from the fact that $\rg \mb P$ is spanned
by eigenfunctions of $\mb L$ with eigenvalues $0$ and $1$ (Proposition \ref{prop:spec}
and Lemma \ref{lem:dimP}).
\end{proof}

\section{Nonlinear perturbation theory}

In this section we consider the full problem Eq.~\eqref{eq:css}, 
\begin{align}
\label{eq:css1}
\partial_0 \phi_1&=-\xi^j \partial_j \phi_1-\phi_1+\phi_2 \nonumber \\
\partial_0 \phi_2&=\partial_j \partial^j \phi_1-\xi^j \partial_j \phi_2-2\phi_2
+6\phi_1+3\sqrt{2}\phi_1^2+\phi_1^3
\end{align}
with prescribed initial data at $\tau=0$.
An operator formulation of Eq.~\eqref{eq:css1} is obtained by defining the nonlinearity
\[
\mb N(\mb u):=\left (\begin{array}{c}
0 \\ 3\sqrt 2 u_1^2+u_1^3 \end{array} \right ).
\]
It is an immediate consequence of the Sobolev embedding $H^1(B)\hookrightarrow L^p(B)$,
$p \in [1,6]$,
that $\mb N: \mc H \to \mc H$ and we have the estimate
\begin{equation}
\label{eq:LipN} 
\|\mb N(\mb u)-\mb N(\mb v)\|\lesssim \|\mb u-\mb v\|(\|\mb u\|+\|\mb v\|) 
\end{equation}
for all $\mb u, \mb v \in \mc H$ with $\|\mb u\|, \|\mb v\|\leq 1$.
The Cauchy problem for Eq.~\eqref{eq:css1} is formally equivalent to
\begin{align}
\label{eq:cssop}
\tfrac{d}{d\tau}\Phi(\tau)&=\mb L\Phi(\tau)+\mb N(\Phi(\tau)) \nonumber \\
\Phi(0)&=\mb u
\end{align}
for a strongly differentiable function $\Phi: [0,\infty)\to \mc H$
where $\mb u$ are the prescribed data.
In fact, we shall consider the weak version of Eq.~\eqref{eq:cssop} which reads
\begin{equation}
\label{eq:cssopweak}
\Phi(\tau)=\mb S(\tau)\mb u+\int_0^\tau \mb S(\tau-\sigma)\mb N(\Phi(\sigma))d\sigma.
\end{equation}
Since the semigroup $\mb S$ is unstable, one cannot expect to obtain a global solution of
Eq.~\eqref{eq:cssopweak} for general data $\mb u\in \mc H$.
However, on the subspace $\ker \mb P$, the semigroup $\mb S$ is stable (Proposition
\ref{prop:lin}). 
In order to isolate the instability in the nonlinear context, we formally project
Eq.~\eqref{eq:cssopweak} to the unstable subspace $\rg \mb P$ which yields
\[ \mb P \Phi(\tau)=\mb S(\tau)\mb P \mb u+\int_0^\tau 
\mb S(\tau-\sigma) \mb P\mb N(\Phi(\sigma))d\sigma. \]
This suggests to subtract the ``bad'' term
\[ \mb S(\tau)\mb P\mb u+\int_0^\infty \mb S(\tau-\sigma)\mb P\mb N(\Phi(\sigma))d\sigma \]
from Eq.~\eqref{eq:cssopweak} in order to force decay.
We obtain the equation
\begin{equation}
\label{eq:cssopweakmod}
\Phi(\tau)=\mb S(\tau)(1-\mb P)\mb u+\int_0^\tau \mb S(\tau-\sigma)\mb N(\Phi(\sigma))d\sigma
-\int_0^\infty \mb S(\tau-\sigma)\mb P\mb N(\Phi(\sigma))d\sigma.
\end{equation}
First, we solve Eq.~\eqref{eq:cssopweakmod} and then we relate solutions of 
Eq.~\eqref{eq:cssopweakmod} to solutions of Eq.~\eqref{eq:cssopweak}.

\subsection{Solution of the modified equation}

We solve Eq.~\eqref{eq:cssopweakmod} by a fixed point argument.
To this end we define 
\begin{align*} 
\mb K_{\mb u}(\Phi)(\tau):=&\mb S(\tau)(1-\mb P)\mb u
+\int_0^\tau \mb S(\tau-\sigma)\mb N(\Phi(\sigma))d\sigma \\
&-\int_0^\infty \mb S(\tau-\sigma)\mb P\mb N(\Phi(\sigma))d\sigma.
\end{align*}
and show that $\mb K_\mb u$ defines a contraction mapping on (a closed subset of) 
the Banach space $\mc X$, given by
\[ \mc X:=\left \{\Phi \in C([0,\infty),\mc H): \sup_{\tau>0} e^{(\frac12-\epsilon)\tau}
\|\Phi(\tau)\|<\infty \right \} \]
with norm
\[ \|\Phi\|_{\mc X}:=\sup_{\tau>0} e^{(\frac12-\epsilon)\tau}
\|\Phi(\tau)\| \]
where $\epsilon \in (0,\frac12)$ is arbitrary but fixed.
We further write
\[ \mc X_\delta:=\{\Phi \in \mc X: \|\Phi\|_{\mc X}\leq \delta\} \]
for the closed ball of radius $\delta>0$ in $\mc X$.

\begin{proposition}
\label{prop:K}
Let $\delta>0$ be sufficiently small and suppose $\mb u \in \mc H$ with $\|\mb u\|<\delta^2$.
Then $\mb K_{\mb u}$ maps $\mc X_\delta$ to $\mc X_\delta$ and we have the estimate
\[ \|\mb K_{\mb u}(\Phi)-\mb K_{\mb u}(\Psi)\|_{\mc X}\leq C\delta \|\Phi-\Psi\|_{\mc X} \]
for all $\Phi, \Psi \in \mc X_\delta$.
\end{proposition}

\begin{proof}
First observe that $\mb K_{\mb u}: \mc X_\delta \to C([0,\infty),\mc H)$ since
$\|\mb N(\Phi(\tau))\|\lesssim e^{(-1+2\epsilon)\tau}$ for any $\Phi \in \mc X_\delta$.
We have
\begin{align}
\mb P \mb K_{\mb u}(\Phi)(\tau)&=-\int_\tau^\infty \mb S(\tau-\sigma)\mb P \mb N(\Phi(\sigma))d\sigma
\end{align}
which yields
\begin{align*} 
\|\mb P [\mb K_{\mb u}(\Phi)(\tau)-\mb K_{\mb u}(\Psi)(\tau)]\|&\lesssim
\int_\tau^\infty e^{\tau-\sigma}\|\Phi(\sigma)-\Psi(\sigma)\|(\|\Phi(\sigma)\|+\|\Psi(\sigma)\|)d\sigma \\
&\lesssim \|\Phi-\Psi\|_{\mc X}(\|\Phi\|_{\mc X}+\|\Psi\|_{\mc X})e^\tau 
\int_\tau^\infty e^{(-2+2\epsilon)\sigma}d\sigma \\
&\lesssim \delta e^{(-1+2\epsilon)\tau}\|\Phi-\Psi\|_{\mc X}
\end{align*}
for all $\Phi,\Psi \in \mc X_\delta$ by Proposition \ref{prop:lin}.
On the stable subspace we have
\[ (1-\mb P)\mb K_{\mb u}(\Phi)(\tau)=\mb S(\tau)(1-\mb P)\mb u+\int_0^\tau
\mb S(\tau-\sigma)(1-\mb P)\mb N(\Phi(\sigma))d\sigma \]
and thus,
\begin{align*} 
\|(1-\mb P)[\mb K_{\mb u}(\Phi)(\tau)-\mb K_{\mb u}(\Psi)(\tau)]\|&\lesssim
\int_0^\tau e^{(-\frac12+\epsilon)(\tau-\sigma)}\|\Phi(\sigma)-\Psi(\sigma)\| \\
&\quad \times (\|\Phi(\sigma)\|+\|\Psi(\sigma)\|)d\sigma \\
&\lesssim \|\Phi-\Psi\|_{\mc X}\delta e^{(-\frac12+\epsilon)\tau}\
\int_0^\tau e^{(-\frac12+\epsilon)\sigma}d\sigma \\
&\lesssim \delta e^{(-\frac12+\epsilon)\tau}\|\Phi-\Psi\|_{\mc X}
\end{align*}
again by Proposition \ref{prop:lin}.
We conclude that
\[ \|\mb K_{\mb u}(\Phi)-\mb K_{\mb u}(\Psi)\|_{\mc X}\lesssim \delta \|\Phi-\Psi\|_{\mc X} \]
for all $\Phi, \Psi \in \mc X$.
By a slight modification of the above argument one similarly 
proves $\|\mb K_{\mb u}(\Phi)\|_{\mc X}\leq \delta$ 
for all $\Phi \in \mc X_\delta$ (here $\|\mb u\|\leq \delta^2$
is used).
\end{proof}

Now we can conclude the existence of a solution to Eq.~\eqref{eq:cssopweakmod} by invoking
the contraction mapping principle.

\begin{lemma}
\label{lem:solmod}
Let $\delta>0$ be sufficiently small. Then, for any $\mb u\in \mc H$ with $\|\mb u\|\leq \delta^2$,
there exists a unique solution $\Phi_\mb u \in \mc X_\delta$
to Eq.~\eqref{eq:cssopweakmod}. 
\end{lemma}

\begin{proof}
By Proposition \ref{prop:K} we may choose $\delta>0$ so small that
\[ \|\mb K_\mb u(\Phi)-\mb K_{\mb u}(\Psi)\|_{\mc X}\leq \tfrac12 \|\Phi-\Psi\|_{\mc X} \]
for all $\Phi,\Psi \in \mc X_\delta$
and thus, the contraction mapping principle implies the existence
of a unique $\Phi_{\mb u} \in \mc X_\delta$ with $\Phi_{\mb u}=\mb K_{\mb u}(\Phi_{\mb u})$. 
By the definition of
$\mb K_{\mb u}$, $\Phi_{\mb u}$ is a solution
to Eq.~\eqref{eq:cssopweakmod}.
\end{proof}

\subsection{Solution of Eq.~\eqref{eq:cssopweak}}
Recall that $\rg \mb P$ is spanned by eigenfunctions of $\mb L$ with eigenvalues $0$ and $1$,
see Lemma \ref{lem:dimP}.
As in the proof of Lemma \ref{lem:dimP} we write $\mb P=\mb P_0+\mb P_1$
where $\mb P_n$, $n\in \{0,1\}$,
 projects to the geometric eigenspace of $\mb L$ associated to the eigenvalue 
$n\in \sigma_p(\mb L)$.
Consequently, we infer $\mb S(\tau)\mb P_n=e^{n\tau}\mb P_n$.
This shows that the ``bad'' term we subtracted from Eq.~\eqref{eq:cssopweak} may be
written as
\begin{align*} \mb S(\tau)\mb P \mb u
+\int_0^\infty \mb S(\tau-\sigma)\mb P\mb N(\Phi_{\mb u}(\sigma))d\sigma
=\mb S(\tau)[\mb P\mb u-\mb F(\mb u)]
\end{align*}
where $\mb F$ is given by
\[ \mb F(\mb u):=-\mb P_0\int_0^\infty \mb N(\Phi_{\mb u}(\sigma))d\sigma 
-\mb P_1 \int_0^\infty e^{-\sigma}\mb N(\Phi_{\mb u}(\sigma))d\sigma. \]
According to Lemma \ref{lem:solmod}, the function $\mb F$ is well-defined on 
$\mc B_{\delta^2}:=\{\mb u\in \mc H: \|\mb u\|< \delta^2\}$
with values in $\rg \mb P$
and this shows that we have effectively modified the \emph{initial data} by adding an element
of the 4-dimensional subspace $\rg \mb P$ of $\mc H$.
Note, however, that the modification depends on the solution itself.
Consequently, if the initial data for Eq.~\eqref{eq:cssopweak} are of the form
$\mb u+\mb F(\mb u)$ for $\mb u\in \ker \mb P$, Eq.~\eqref{eq:cssopweak} and \eqref{eq:cssopweakmod}
are equivalent and Lemma \ref{lem:solmod} yields the desired solution of Eq.~\eqref{eq:cssopweak}.
We also remark that $\mb F(\mb 0)=\mb 0$.
The following result implies that the graph 
\[ \{\mb u+\mb F(\mb u): \mb u \in \ker \mb P, \|\mb u\|<\delta^2\} \subset \ker \mb P
\oplus \rg \mb P= \mc H \]
defines a Lipschitz manifold of co-dimension $4$.

\begin{lemma}
\label{lem:Lip}
Let $\delta>0$ be sufficiently small. Then the function $\mb F: \mc B_{\delta^2}
\to \rg \mb P\subset \mc H$ satisfies
\[ \|\mb F(\mb u)-\mb F(\mb v)\|\leq C\delta \|\mb u-\mb v\|. \]
\end{lemma}

\begin{proof}
First, we claim that $\mb u\mapsto \Phi_\mb u: \mc B_{\delta^2}\to \mc X_\delta
\subset \mc X$ is Lipschitz-continuous.
Indeed, since $\Phi_\mb u=\mb K_{\mb u}(\Phi_{\mb u})$ we infer
\begin{align*} \|\Phi_{\mb u}-\Phi_{\mb v}\|_{\mc X}&\leq 
\|\mb K_{\mb u}(\Phi_{\mb u})-\mb K_{\mb u}(\Phi_{\mb v})\|_{\mc X}
+ \|\mb K_{\mb u}(\Phi_{\mb v})-\mb K_{\mb v}(\Phi_{\mb v})\|_{\mc X} \\
&\lesssim \delta \|\Phi_{\mb u}-\Phi_{\mb v}\|_{\mc X}+\|\mb u-\mb v\|
\end{align*}
by Proposition \ref{prop:K} and the fact that
\[ \|\mb K_{\mb u}(\Phi_{\mb v})(\tau)-\mb K_{\mb v}(\Phi_{\mb v})(\tau)\|
=\|\mb S(\tau)(1-\mb P)(\mb u-\mb v)\|\lesssim e^{(-\frac12+\epsilon)\tau}\|\mb u-\mb v\|. \]
The claim now follows from $\|\mb N(\mb u)-\mb N(\mb v)\|\lesssim \|\mb u-\mb v\|(\|\mb u\|
+\|\mb v\|)$.
\end{proof}

We summarize our results in a theorem.

\begin{theorem}
\label{thm:main2}
Let $\delta>0$ be sufficiently small.
There exists a co-dimension 4 Lipschitz manifold $\mc M\subset \mc H$ with $\mb 0\in \mc M$
such that for
any $\mb u\in \mc M$, Eq.~\eqref{eq:cssopweak} has a solution $\Phi \in \mc X_\delta$.
Moreover, $\Phi$ is unique in $C([0,\infty),\mc H)$.
If, in addition, $\mb u\in \mc D(\mb L)$ then $\Phi \in C^1([0,\infty),\mc H)$ and $\Phi$
solves Eq.~\eqref{eq:cssop} with $\Phi(0)=\mb u$.
\end{theorem}

\begin{proof}
The last statement follows from standard results of semigroup theory and uniqueness
in $C([0,\infty),\mc H)$ is a simple exercise. 
\end{proof}

\subsection{Proof of Theorem \ref{thm:main}}
Theorem \ref{thm:main} is now a consequence of 
Theorem \ref{thm:main2}: Eq.~\eqref{eq:vphi} implies
\begin{equation}
\label{eq:vphi2} \frac{v\circ \Phi_T(X)-v_0\circ \Phi_T(X)}{T^2-|X|^2}=\tfrac{1}{(-T)}\phi(-\log(-T),
\tfrac{X}{(-T)}) 
\end{equation}
and thus,
\begin{align*} |T|^{-1}\|v-v_0\|_{L^2(\Sigma_T)}&=|T|^{-2}\left \|
\phi_1\left (-\log(-T),\tfrac{\cdot}{|T|}\right ) \right \|_{L^2(B_{|T|})} \\
&=|T|^{-\frac12} \|\phi_1(-\log(-T),\cdot)\|_{L^2(B)} \\
&\lesssim |T|^{-\epsilon}.
\end{align*}
Similarly, we obtain
\[ \partial_{X^j}\frac{v\circ \Phi_T(X)-v_0\circ \Phi_T(X)}{T^2-|X|^2}=
\tfrac{1}{T^2}\partial_j \phi_1(-\log(-T),
\tfrac{X}{(-T)}) \] 
which yields
\[ \|v-v_0\|_{\dot H^1(\Sigma_T)}=T^{-2}\left \|
\nabla \phi_1\left (-\log(-T),\tfrac{\cdot}{|T|}\right ) \right \|_{L^2(B_{|T|})}
\lesssim |T|^{-\epsilon}. \]
For the time derivative we infer
\begin{align*} \partial_T\frac{v\circ \Phi_T(X)-v_0\circ \Phi_T(X)}{T^2-|X|^2}&=\tfrac{1}{T^2}
\left (\partial_0 \phi+\tfrac{X^j}{(-T)}\partial_j \phi+\phi \right )
\left (-\log(-T),\tfrac{X}{(-T)}\right ) \\
&=\tfrac{1}{T^2}\phi_2 \left (-\log(-T),\tfrac{X}{(-T)}\right )
\end{align*}
and hence,
\[ \|\nabla_n v-\nabla_n v_0\|_{L^2(\Sigma_T)}=T^{-2} 
\left \|\phi_2 \left (-\log(-T),\tfrac{\cdot}{|T|} \right ) \right \|_{L^2(B_{|T|})}
\lesssim |T|^{-\epsilon}. \]

Finally, we turn to the Strichartz estimate.
First, note that the 
modulus of the determinant of the Jacobian of $(T,X) \mapsto (t,x)$
is $(T^2-|X|^2)^{-4}$.
This is easily seen by considering the transformation
\[ X^\mu \mapsto y^\mu=-\frac{X^\mu}{X_\sigma X^\sigma} \]
which has the same Jacobian determinant (up to a sign) since
$t=-y^0$ and $x^j=y^j$.
We obtain
\[ \partial_\nu y^\mu=-\frac{X_\sigma X^\sigma \delta_\nu{}^\mu-2X_\nu X^\mu}
{(X_\sigma X^\sigma)^2} \]
and hence,
\[ \partial_\nu y^\mu \partial_\mu y^\lambda=\frac{\delta_\nu{}^\lambda}{(X_\sigma X^\sigma)^2} \]
which yields $|\det(\partial_\nu y^\mu)|=(X_\sigma X^\sigma)^{-4}=(T^2-|X|^2)^{-4}$.
Furthermore, note that $s\in [t,2t]$ and $x\in B_{(1-\delta)t}$ imply
\begin{align*} 
S:&=-\frac{s}{s^2-|x|^2}\geq -\frac{t}{t^2-|x|^2}\geq -\frac{c_\delta}{t} \\
S&\leq -\frac{2t}{4t^2-|x|^2}\leq -\frac{1}{2t}.
\end{align*}
Consequently, by Eq.~\eqref{eq:vphi2} and Sobolev embedding we infer
\begin{align*}
\|v-v_0\|_{L^4(t,2t)L^4(B_{(1-\delta)t})}^4&\leq \int_{-\frac{c_\delta}{t}}^{-\frac{1}{2t}}
\int_{B_{(1-\delta)|S|}}\left |
\frac{v\circ \Phi_S(X)
-v_0 \circ \Phi_S(X)}{S^2-|X|^2} \right |^4 dX dS \\
&\lesssim \int_{-\frac{c_\delta}{t}}^{-\frac{1}{2t}}|S|^{-4} \|\phi(-\log(-S),\tfrac{\cdot}{|S|})
\|_{L^4(B_{|S|})}^4 dS \\
&= \int_{-\frac{c_\delta}{t}}^{-\frac{1}{2t}}|S|^{-1} \|\phi(-\log(-S),\cdot)
\|_{L^4(B)}^4 dS \\
&\lesssim  \int_{-\frac{c_\delta}{t}}^{-\frac{1}{2t}}|S|^{-1} \|\phi(-\log(-S),\cdot)
\|_{H^1(B)}^4 dS \\
&\lesssim \int_{-\frac{c_\delta}{t}}^{-\frac{1}{2t}} |S|^{1-4\epsilon}dS \simeq t^{-2+4\epsilon}
\end{align*}
as claimed.

\bibliography{cubic_nosym}
\bibliographystyle{plain}

\end{document}